\documentclass[leqno,10pt]{amsart}
\usepackage{amsmath,amstext,amssymb,amsopn,amsthm,mathrsfs,color}
\usepackage{enumerate}
\usepackage{graphicx}
\usepackage{multirow}

\allowdisplaybreaks

\DeclareMathOperator*{\essup}{ess\,sup}

\DeclareMathOperator{\domain}{Dom}

\DeclareMathOperator{\diam}{diam}
\DeclareMathOperator{\id}{Id}

\newtheorem{thm}{Theorem}[section]
\newtheorem*{thm*}{Theorem}

\newtheorem{propo}[thm]{Proposition}
\newtheorem{lem}[thm]{Lemma}

\newtheorem*{rem*}{Remark}

\newcommand{\R}{\mathbb{R}}

\def\a{\alpha}
\def\b{\beta}
\def\ab{\alpha,\beta}
\def\t{\theta}
\def\v{\varphi}
\def\s{\sigma}

\begin{document}

\author[A. Nowak]{Adam Nowak}
\address{Instytut Matematyczny\\
Polska Akademia Nauk\\
\'Sniadeckich 8\\
00-956 Warszawa, Poland}
\email{adam.nowak@impan.pl}

\author[L. Roncal]{Luz Roncal}
\address{Departamento de Matem\'aticas y Computaci\'on\\
Universidad de La Rioja\\
26004 Logro\~no, Spain}
\email{luz.roncal@unirioja.es}

\footnotetext{
		The first-named author was supported in part by 
		a grant from the National Science Centre of Poland.
		Research of the second-named author supported by the grant
		MTM2012-36732-C03-02 from the Spanish Government.
}

\title[Potential operators]
	{Potential operators \\associated with Jacobi and Fourier-Bessel expansions} 

\begin{abstract}
We study potential operators (Riesz and Bessel potentials) associated with classical Jacobi and Fourier-Bessel
expansions. We prove sharp estimates for the corresponding potential kernels. Then we characterize
those $1\le p,q\le \infty$, for which the potential operators are of strong type $(p,q)$, of weak
type $(p,q)$ and of restricted weak type $(p,q)$. These results may be thought of as analogues
of the celebrated Hardy-Littlewood-Sobolev fractional integration theorem in the Jacobi and Fourier-Bessel
settings. As an ingredient of our line of reasoning, we also obtain sharp estimates of the Poisson kernel
related to Fourier-Bessel expansions.
\end{abstract}

\maketitle

\section{Introduction}

The classical fractional integral operator (also referred to as the {Riesz potential}) is given by
$$
\boldsymbol{I}_{\s} f(x) = \int_{\mathbb{R}^d} \|x-y\|^{2\s-d} f(y)\, dy, \qquad x \in \mathbb{R}^d;
$$
here $d\ge 1$ and $0< \s < d/2$. The integral defining $\boldsymbol{I}_{\s}$ converges for
$\textrm{a.a.}\;x \in \mathbb{R}^d$ provided that $f\in L^p(\mathbb{R}^d)$ and $\frac{1}p > \frac{2\s}{d}$.
For sufficiently smooth functions $f$, $\boldsymbol{I}_{\s}$ coincides up to a constant factor with
$(-\Delta)^{-\s}$, where $\Delta$ is the standard Laplacian in $\mathbb{R}^d$ and the negative
power is defined by means of the Fourier transform.
A basic result concerning mapping properties of $\boldsymbol{I}_{\s}$ 
is the Hardy-Littlewood-Sobolev theorem, see e.g.~\cite[Chapter V]{Stein}, which says that
$\boldsymbol{I}_{\s}$ is of weak type $(1,\frac{d}{d-2\s})$ and of strong type $(p,q)$ when
$\frac{1}q = \frac{1}p -\frac{2\s}d$ and $\quad p>1$ and $q< \infty$.
One can also check that $\boldsymbol{I}_{\s}$ is of restricted weak type $(\frac{d}{2\s},\infty)$ and that
these mapping properties are sharp in the sense that $\boldsymbol{I}_{\s}$ is not of strong type
$(1,\frac{d}{d-2\s})$, not of weak type $(\frac{d}{2\s},\infty)$ and not of restricted weak type $(p,q)$
when $\frac{1}q \neq \frac{1}p-\frac{2\s}d$ (see \cite[Chapter V]{SW} for the terminology).

Weighted estimates for $\boldsymbol{I}_{\s}$ with power weights were studied in \cite{St-We}, 
and with general weights by several authors, see for example \cite{Sa-Wh} and references therein.
On the other hand, numerous analogues of $\boldsymbol{I}_{\s}$ have been investigated in various settings,
including metric measure spaces, spaces of homogeneous type, orthogonal expansions, etc.,
see e.g.~\cite{AM,Bernardis-Salinas,BoTo,GST,GT,GattoSegovia,K,KS,Mu,MuSt,NoSt1},
and references in these papers.
$L^p-L^q$ estimates for such operators are of interest, for instance, in the study of higher order
Riesz transforms and Sobolev spaces in the above mentioned contexts. 
Recently some $L^p-L^q$ bounds for the potential operator in the context of Jacobi expansions, as well
as vector-valued extensions for that operator, were obtained in \cite{OSRS}.
Another recent result in this spirit can be found in \cite{NoSt2}, where a sharp description of
$L^p-L^q$ mapping properties of the potential operator associated with the harmonic
oscillator (the setting of classical Hermite function expansions) was established.

The aim of this paper is to obtain similar characterizations of $L^p-L^q$ mapping properties
of potential operators in the contexts of classical Jacobi and Fourier-Bessel expansions.
These settings are in principle one dimensional, but have a geometric multi-dimensional background.
More precisely, the Fourier-Bessel framework with half-integer parameters of type $\nu$
is related to analysis in the
Euclidean unit balls of dimensions $d_{\nu}=2\nu+2$, see \cite[Chapter 2, H]{Folland} and \cite{NR1}. 
In the Jacobi setting there are two
parameters of type, $\alpha$ and $\beta$. When they are equal and half-integer,
the Jacobi context is related to analysis
on the Euclidean unit spheres of dimensions $d_{\alpha}= 2\alpha+2 = 2\beta+2 = d_{\beta}$, see
\cite[Section 3]{NS2}. If the half-integer parameters are different, say $\alpha> \beta$,
then the geometric connection is more
complex and involves the unit spheres of dimensions $d_{\alpha}$
and unit discs of dimensions $d_{\beta}$ inside those spheres, see \cite{B}.
The `geometric' dimensions $d_{\nu}$, $d_{\alpha}$ and $d_{\beta}$ manifest in our results, and
the interplay between them and the `physical' dimension $d=1$ is an interesting and important
aspect of the theorems dealing with mapping properties of the potential operators.

The main results of the paper are characterizations of those $1\le p,q \le \infty$, for which
the Jacobi and Fourier-Bessel potential operators are of strong type $(p,q)$, of weak type $(p,q)$
and of restricted weak type $(p,q)$, see Theorems \ref{thm:LpLqJac}, \ref{thm:LpLqJacL},
\ref{thm:LpLqFB}, \ref{thm:LpLqFBL} in Section \ref{sec:prel}. Comparing to the
Hardy-Littlewood-Sobolev theorem, the Jacobi and Fourier-Bessel potential operators possess
better mapping properties. This is, roughly speaking, thanks to the finiteness of the measures involved,
and also due to the fact that spectra of the Jacobi and Fourier-Bessel `Laplacians', being discrete,
are separated from $0$
(assuming in the Jacobi case that the parameters of type satisfy $\alpha+\beta\neq -1$).
The proofs are based on sharp estimates for the corresponding potential kernels, which we also obtain
in this paper, see Section \ref{sec:prel}. The latter depend on sharp estimates for the associated
Poisson kernels, which in the Jacobi case were found recently in \cite{NS2,NSS}, and for the
Fourier-Bessel case are established in this paper (Theorem \ref{thm:FBPest}).

The frameworks considered in this paper can be described in a unified and more general way as follows.
Let $X \subset \mathbb{R}$ be a finite interval equipped with a measure $\mu$.
Let $\{\v_n : n \ge N\}$, $N=0$ or $N=1$, be an orthonormal basis of $L^2(X,d\mu)$ consisting of
eigenfunctions of a second order differential operator $L$ (a `Laplacian'),
$$
L \v_n = \lambda_n \v_n, \qquad n \ge N.
$$
In each of our settings $\lambda_n$ are nonnegative, increasing in $n$,
have multiplicities $1$, and $\lambda_n \simeq n^2$
as $n \to \infty$. Moreover, there is a self-adjoint extension of $L$, still denoted by the same symbol,
whose spectral resolution is given by the $\v_n$ and $\lambda_n$.

Assuming that the bottom eigenvalue is nonzero, we consider the \emph{negative powers}
$$
L^{-\sigma} f = \sum_{n \ge N} \lambda_n^{-\sigma} {\langle f,\v_n\rangle_{\mu}}\, \v_n,
$$
where $\s > 0$. For $f \in L^2(d\mu)$ the above series converges in $L^2(d\mu)$ and defines a bounded
linear operator. Formally, $L^{-\s}$ can be written as an integral operator, which we denote by $I_{\s}$,
\begin{equation} \label{potop}
I_{\s}f(x) = \int_{X} K_{\s}(x,y) f(y)\, d\mu(y),
\end{equation}
where the kernel $K_{\s}(x,y)$ can be expressed by the associated heat kernel
or any kernel subordinated to it. In particular, if
$$
H_t(x,y) = \sum_{n \ge {N}} \exp\big(-t {\lambda_n^{1/2}}\big) \,\v_n(x) \v_n(y)
$$
is the corresponding Poisson kernel (the kernel of the semigroup $\{\exp(-t L^{1/2})\}$), then
\begin{equation} \label{pk}
K_{\s}(x,y) = \frac{1}{\Gamma(2\s)} \int_0^{\infty} H_t(x,y) t^{2\s-1}\, dt.
\end{equation}
The set of all $f$ for which the integral in \eqref{potop} converges for $\textrm{a.a.}\;x \in X$ forms
$\domain I_{\s}$, the \emph{natural domain} of $I_{\s}$.
We call $I_{\s}$ the \emph{potential operator} and $K_{\s}(x,y)$ the \emph{potential kernel}.
In the contexts we study, $K_{\s}(x,y)$ is always well defined by \eqref{pk} for $x \neq y$,
and \eqref{potop} makes sense for a large class of $f$.
Furthermore, by our results and arguments similar to those in the
proof of \cite[Corollary 2.4]{NoSt1}, it can be verified that $L^{-\s}$ and $I_{\s}$ coincide as
operators on $L^2(d\mu)$. The integral representation \eqref{potop} of the potential operator offers
an intrinsic and direct approach to the negative powers of $L$. In particular, it enables us to
describe, in a sharp way, $L^p-L^q$ mapping properties of $I_{\s}$ in all the investigated settings.
More general weighted $L^p-L^q$ results are also possible, but are beyond the scope of this paper.

Some remarks are in order. First of all, note that philosophically it would be more appropriate to define
in our settings $K_{\s}(x,y)$ via the heat kernels, i.e.~the kernels of the semigroups $\{\exp(-tL)\}$.
However, although qualitatively sharp estimates of the Jacobi and Fourier-Bessel
heat kernels are available, see \cite{CKP,NR1,NR2,NS2},
from the analytic point of view of estimating $K_{\s}(x,y)$, it seems more convenient to use $H_t(x,y)$
since no exponential factors are needed to describe its short time behavior. 

Another comment concerns terminology. In the literature devoted to analysis of orthogonal
expansions it often happens that the phrase \emph{fractional integral} refers not only to negative
powers, but also to multiplier operators given either by
$$
f \mapsto \sum_{n \ge 1} n^{-\s} {\langle f, \v_n\rangle_{\mu}}\, \v_n
$$
or by
$$
f \mapsto \sum_{n \ge 0} (n+1)^{-\s} {\langle f, \v_n\rangle_{\mu}}\, \v_n,
$$
see for instance \cite[Chapter III]{MuSt} or comments throughout \cite{NoSt1} and references given there.
These definitions differ in various particular contexts (including those considered by us) from the
negative powers of $L$ defined spectrally.

On the other hand, some mapping properties of the fractional integrals above and the negative powers are
related by means of suitable multiplier theorems, see for instance \cite{DeNapoliDrelichmanDuran,GST,GT,KS}
and the ends of Sections 2--4 in \cite{NoSt1}.
Even more, the negative powers and the fractional integrals can be treated directly by means of
multiplier theorems. Still another possibility of dealing with these operators in some particular settings
is based on transplantation theorems, see \cite[p.\,213]{NoSt1} for some hints on the idea.
All these multiplier and transplantation aspects pertain only $L^p-L^p$ mapping properties and will
not be further discussed here.
For the settings investigated in this paper, the reader can find suitable multiplier and transplantation
theorems for instance in \cite{OSNS,OS1,OS2,L,Mu,NSS}.
A multiplier approach to fractional integrals in a Jacobi framework slightly different from those
considered here can be found in \cite{BaUr1,BaUr2}.

Finally, we note that the results of this paper concerning the analogues of the classical Riesz potentials
in the Jacobi and Fourier-Bessel settings contain implicitly parallel results for counterparts
of the classical Bessel potentials $(\id-\Delta)^{-\sigma}$. More precisely, in each of our frameworks
we may consider the negative powers $(\id + L)^{-\sigma}$, which can be written as integral operators
$$
\widetilde{I}_{\sigma}f(x) = \int_{X} \widetilde{K}_{\sigma}(x,y)f(y)\, d\mu(y).
$$
Notice that $(\id+L)^{-\sigma}$ makes sense spectrally also in cases when $0$ is an eigenvalue of $L$.
Since the heat kernels related to $L$ and $\id+L$ coincide up to the factor $\exp(-t)$, the arguments
proving Theorem \ref{thm:FBPest} show that the corresponding Poisson kernels are comparable,
uniformly in $\t,\v$ and $t$, up to the factor $\exp(-t [(\lambda_N+1)^{1/2}-(\lambda_N)^{1/2}])$.
Thus the reasonings proving Theorems \ref{thm:potker} and \ref{thm:potkerFB} go through revealing that
these results hold with $K_{\sigma}$ replaced by $\widetilde{K}_{\sigma}$ in each case and with the
restriction $\alpha+\beta\neq -1$ in Theorem \ref{thm:potker} released. Consequently, Theorems
\ref{thm:LpLqJac}, \ref{thm:LpLqJacL}, \ref{thm:LpLqFB}, \ref{thm:LpLqFBL} still hold after replacing
$I_{\sigma}$ by $\widetilde{I}_{\sigma}$ in each setting, moreover with the restriction 
$\alpha+\beta \neq -1$ removed in cases of Theorems \ref{thm:LpLqJac} and \ref{thm:LpLqJacL}.

The paper is organized as follows. In Section \ref{sec:prel} we briefly introduce the Jacobi and
Fourier-Bessel settings to be investigated and state the main results
(Theorems \ref{thm:potker}-\ref{thm:LpLqFBL}). The corresponding proofs
are contained in the two succeeding sections. In Section \ref{sec:ker} we show sharp estimates for
all the relevant potential kernels (Theorems \ref{thm:potker} and \ref{thm:potkerFB}),
and also for the Poisson kernel associated with Fourier-Bessel
expansions (Theorem \ref{thm:FBPest}). 
Finally, Section \ref{sec:LpLq} is devoted to proving $L^p-L^q$ mapping properties
of the Jacobi and Fourier-Bessel potential operators (Theorems \ref{thm:LpLqJac}, \ref{thm:LpLqJacL},
\ref{thm:LpLqFB} and \ref{thm:LpLqFBL}).

\textbf{Notation.} Throughout the paper we use a standard notation.
In particular, by $\langle f,g \rangle_{\mu}$ we mean $\int f(x) \overline{g(x)} \, d\mu(x)$
whenever the integral makes sense.
For $1\le p \le \infty$, $p'$ is its adjoint, $1/p+1/p'=1$.
When writing estimates, we will frequently use the notation $X \lesssim Y$
to indicate that $X \le CY$ with a positive constant $C$ independent of significant quantities.
We shall write $X \simeq Y$ when simultaneously $X \lesssim Y$ and $Y \lesssim X$.
For the sake of clarity and reader's convenience, in the Appendix we include a table summarizing
the notation of various objects in the settings considered in this paper.

\section{Preliminaries and statement of results} \label{sec:prel}

We will consider two interrelated settings of orthogonal systems based on Jacobi polynomials. 
Also, we will study two contexts of Fourier-Bessel expansions, which are close (in a sense to
be explained in Section \ref{sec:ker}) to the two Jacobi setting with parameters of type 
$\alpha=\nu$ and $\beta=1/2$. All the four settings have roots in the existing literature.
\subsection{Jacobi trigonometric polynomial setting}
Let $\ab > -1$. The normalized trigonometric Jacobi polynomials are given by
$$
\mathcal{P}_n^{\ab}(\t) = c_n^{\ab} P_n^{\ab}(\cos\t), \qquad \t \in (0,\pi),
$$
where $c_n^{\ab}$ are normalizing constants, and $P_n^{\ab}$, $n\ge 0$, are the classical 
Jacobi polynomials as defined in Szeg\"o's monograph \cite{Sz}; see \cite{NS1,NS2,NSS}.
The system $\{\mathcal{P}_n^{\ab}: n \ge 0\}$ is an orthonormal basis in $L^2(d\mu_{\ab})$,
where $\mu_{\ab}$ is a measure on the interval $(0,\pi)$ defined by
$$
d\mu_{\ab}(\t) = \Big( \sin\frac{\t}2 \Big)^{2\alpha+1} \Big( \cos\frac{\t}2\Big)^{2\beta+1} d\t.
$$
It consists of eigenfunctions of the Jacobi differential operator
$$
\mathcal{J}^{\ab} = - \frac{d^2}{d\theta^2} - \frac{\alpha-\beta+(\alpha+\beta+1)\cos\theta}{\sin \theta}
	\frac{d}{d\theta} + \Big( \frac{\alpha+\beta+1}{2}\Big)^2;
$$
more precisely,
$$
\mathcal{J}^{\ab} \mathcal{P}_n^{\ab} = \Big( n + \frac{\alpha+\beta+1}{2}\Big)^2 \mathcal{P}_n^{\ab}, 
	\qquad n \ge 0.
$$
We shall denote by the same symbol $\mathcal{J}^{\ab}$ the natural self-adjoint extension whose
spectral resolution is given by the $\mathcal{P}_n^{\ab}$, see \cite[Section 2]{NS1} for details.

The integral kernel $\mathcal{H}_t^{\ab}(\t,\v)$ of the Jacobi-Poisson semigroup 
$\{\exp(-t(\mathcal{J}^{\ab})^{1/2})\}$ can be expressed via a complicated hypergeometric function
of two variables, or by means of a double-integral representation, see \cite[Proposition 4.1]{NS1}
and \cite[Section 2]{NSS}. However, none of these expressions provides a direct view of the behavior
of the kernel. The following sharp estimate of $\mathcal{H}_t^{\ab}(\t,\v)$ was obtained recently in
\cite{NS2} for $\ab \ge -1/2$ and in \cite[Theorem 6.1]{NSS} 
for the remaining $\alpha$ and $\beta$.
\begin{thm}[{\cite{NS2,NSS}}]
\label{thm:estPJ}
Let $\ab > -1$. Given any $T>0$, we have
\begin{equation*}
\mathcal{H}^{\ab}_t(\t,\v)  \simeq \begin{cases}
	( t+ \t+\v )^{-2\a-1}
	( t + 2\pi-\t -\v )^{-2\b-1} \frac{t}{t^2+(\t-\v)^2}, & t \le T \\
	\exp\left( -t \frac{|\a+\b+1|}{2} \right), & t > T
	\end{cases},
\end{equation*}
uniformly in $t>0$ and $\t,\v \in (0,\pi)$.
\end{thm}

Assume that $\alpha+\beta \neq -1$, so that the spectrum of $\mathcal{J}^{\ab}$ is separated from $0$.
Given $\s>0$, consider the potential operator
$$
\mathcal{I}^{\ab}_{\s}f(\t) = \int_0^{\pi} \mathcal{K}_{\s}^{\ab}(\t,\v)f(\v)\, d\mu_{\ab}(\v),
$$
where the potential kernel expresses by the Jacobi-Poisson kernel as
\begin{equation*}
\mathcal{K}_{\s}^{\ab}(\t,\v)=\frac{1}{\Gamma(2\s)}\int_0^{\infty}\mathcal{H}_t^{\ab}(\t,\v) t^{2\s-1}\,dt, 
	\qquad \t,\v \in (0,\pi).
\end{equation*}
An upper bound for $\mathcal{K}_{\s}^{\ab}(\t,\v)$ showing explicit dependence on the parameters of type
was obtained recently in \cite[Theorem 1.3]{OSRS},
under the restrictions $\s< 1/2$, $\alpha \ge -1/2$ and $\beta > -1/2$.
 
Here, making use of Theorem \ref{thm:estPJ}, we will prove the following sharp bounds for 
$\mathcal{K}_{\s}^{\ab}(\t,\v)$.
\begin{thm} \label{thm:potker}
Let $\ab > -1$, $\a+\b \neq -1$, and $\s > 0$ be fixed. The estimate
\begin{align*}
\mathcal{K}_{\sigma}^{\ab}(\t,\v) & \simeq 
	1 + \chi_{\{\s = \a+1\}} \log \frac{2\pi}{\t+\v}
	+ \chi_{\{\s=\b+1\}} \log\frac{2\pi}{2\pi - \t-\v} \\
	&\quad  + (\t+\v)^{2\s-2(\a+1)}(2\pi-\t-\v)^{2\s-2(\b+1)}\begin{cases}
   1, & \s>1/2\\
   \log\frac{(\t+\v)(2\pi-\t-\v)}{|\t-\v|}, &\s=1/2\\
    \left(\frac{(\t+\v)(2\pi-\t-\v)}{|\t-\v|}\right)^{1-2\s}, &\s<1/2
    \end{cases}
\end{align*}
holds uniformly in $\t,\v \in (0,\pi)$.
\end{thm}
Notice that locally, when $\t$ and $\v$ stay separated from the boundary of $(0,\pi)$, the kernel
$\mathcal{K}_{\sigma}^{\ab}(\t,\v)$ behaves like the kernels of classical Riesz (when $0 < \sigma < 1/2)$
and Bessel (for all $\sigma > 0$) potentials; see \cite{AS}. A similar comment
concerns potential kernels in all the other contexts considered in this paper.

From the estimate of Theorem \ref{thm:potker}, it can be seen that
$L^p(d\mu_{\ab})\subset \domain\mathcal{I}_{\s}^{\ab}$ for $1\le p \le \infty$.
Furthermore, Theorem \ref{thm:potker} enables a direct analysis of the potential operator 
$\mathcal{I}_{\s}^{\ab}$. The following result gives sharp description of $L^p-L^q$ mapping properties
of $\mathcal{I}_{\s}^{\ab}$, see also Figure \ref{fig1} below.
\begin{thm} \label{thm:LpLqJac}
Let $\ab > -1$, $\alpha+\beta \neq -1$, $\s>0$ and $1\le p,q \le \infty$.
Set 
$$
\delta:= (\alpha+1)\vee (\beta+1) \vee (1/2).
$$
Then $\mathcal{I}_{\s}^{\ab}$ has the following mapping properties with respect to the measure space
$((0,\pi),d\mu_{\ab})$.
\begin{itemize}
\item[(i)] If $\s > \delta$, then $\mathcal{I}_{\s}^{\ab}$ is of strong type $(p,q)$ for all $p$ and $q$.
\item[(ii)] If $\s = \delta$, then $\mathcal{I}_{\s}^{\ab}$ 
	is of strong type $(p,q)$ for $(p,q)\neq (1,\infty)$, and not of restricted weak type $(1,\infty)$.
\item[(iii)] Assume finally that $\s < \delta$.
	Then $\mathcal{I}_{\s}^{\ab}$ is of strong type $(p,q)$ provided that
	$$
		\frac{1}q \ge \frac{1}p -\frac{\s}{\delta} \quad \textrm{and} \quad (p,q) \notin
			\Big\{\Big(1,\frac{\delta}{\delta-\s}\Big), \Big(\frac{\delta}{\s},\infty\Big)\Big\}.
	$$
	Moreover, $\mathcal{I}_{\s}^{\ab}$ is of weak type $\big(1,\frac{\delta}{\delta-\s}\big)$ and of restricted
	weak type $\big(\frac{\delta}{\s},\infty\big)$.
	
	These results are sharp in the sense that $\mathcal{I}_{\s}^{\ab}$ is not of strong type
  $\big(1,\frac{\delta}{\delta-\s}\big)$, not of weak type $\big(\frac{\delta}{\s},\infty\big)$, and
  not of restricted weak type $(p,q)$ when $\frac{1}q < \frac{1}p - \frac{\s}{\delta}$.
\end{itemize}
\end{thm}
\begin{figure} [ht]
\includegraphics[height=0.5\textwidth]{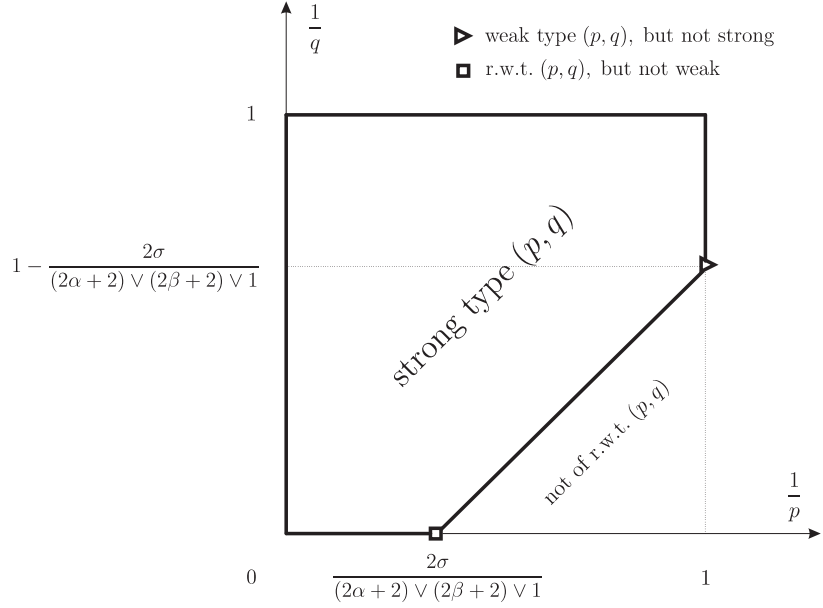}
\caption{Mapping properties of $\mathcal{I}_{\s}^{\ab}$ for $0 < \s < (\alpha+1)\vee (\beta+1) \vee (1/2)$.}
 \label{fig1}
\end{figure}

\subsection{Jacobi trigonometric `function' setting}
This Jacobi setting is derived from the previous one by modifying the Jacobi trigonometric polynomials
so as to make the resulting system orthonormal with respect to Lebesgue measure $d\t$ in $(0,\pi)$.
Thus we consider the functions
\begin{equation} \label{Jtf}
\phi_n^{\ab}(\theta) = \Big( \sin\frac{\theta}2 \Big)^{\alpha+1/2} \Big( \cos\frac{\theta}2\Big)^{\beta+1/2}
	\mathcal{P}_n^{\ab}(\theta), \qquad n \ge 0.
\end{equation}
Then the system $\{\phi_n^{\ab} : n \ge 0\}$ is an orthonormal basis in $L^2(d\theta)$.
The associated differential operator is 
$$
\mathbb{J}^{\ab} = -\frac{d^2}{d\theta^2} + \frac{(\alpha-1/2)(\alpha+1/2)}{4\sin^2\frac{\theta}2}
	+ \frac{(\beta-1/2)(\beta+1/2)}{4\cos^2\frac{\theta}2}
$$
and we have, see \cite[Section 2]{NS2},
$$
\mathbb{J}^{\ab} \phi_n^{\ab} = \Big( n + \frac{\alpha+\beta+1}{2}\Big)^2 \phi_n^{\ab}, \qquad n \ge 0.
$$

The Jacobi-Poisson semigroup $\{\exp(-t(\mathbb{J}^{\ab})^{1/2})\}$, generated by means of the square
root of the natural 
self-adjoint extension of $\mathbb{J}^{\ab}$ in this context, has an integral representation.
The associated integral kernel $\mathbb{H}_t^{\ab}(\t,\v)$ is linked to $\mathcal{H}_t^{\ab}(\t,\v)$
by, see \cite[Section 2]{NS2},
\begin{equation} \label{PPrel}
\mathbb{H}_t^{\ab}(\theta,\varphi) = 
	\Big( \sin\frac{\theta}2 \sin\frac{\varphi}2 \Big)^{\alpha+1/2} 
	\Big( \cos\frac{\theta}2 \cos\frac{\varphi}2 \Big)^{\beta+1/2}
	\mathcal{H}_t^{\ab}(\theta,\varphi).
\end{equation}
Thus Theorem \ref{thm:estPJ} delivers also sharp estimates for $\mathbb{H}_t^{\ab}(\t,\v)$.

Let $\s>0$ and assume that $\alpha+\beta\neq -1$, so that the negative powers $(\mathbb{J}^{\ab})^{-\s}$ are
well defined in $L^2(d\t)$. Consider the potential operator
$$
\mathbb{I}_{\s}^{\ab}f(\t) = \int_0^{\pi} \mathbb{K}_{\s}^{\ab}(\t,\v) f(\v)\, d\v,
$$
where
\begin{align} \nonumber
\mathbb{K}_{\s}^{\ab}(\t,\v) & = 
	\Big( \sin\frac{\theta}2 \sin\frac{\varphi}2 \Big)^{\alpha+1/2} 
	\Big( \cos\frac{\theta}2 \cos\frac{\varphi}2 \Big)^{\beta+1/2}
		\mathcal{K}_{\s}^{\ab}(\t,\v) \\ & \simeq (\t \v)^{\alpha+1/2}\big((\pi-\t)(\pi-\v)\big)^{\beta+1/2} 
		\,\mathcal{K}_{\s}^{\ab}(\t,\v). \label{Kleb}
\end{align}
Clearly, \eqref{Kleb} combined with Theorem \ref{thm:potker} leads to sharp estimates of 
$\mathbb{K}_{\s}^{\ab}(\t,\v)$. Using them it is not hard to see that the natural domain of
$\mathbb{I}_{\s}^{\ab}$ contains all $L^p(d\t)$ spaces, $1\le p \le \infty$, in case $\ab \ge -1/2$.
If $\alpha \wedge \beta < -1/2$, then $L^p(d\t) \subset \domain \mathbb{I}_{\s}^{\ab}$ 
provided that $\frac{1}p < \frac{3}{2}+(\alpha \wedge \beta)$.

The following result gives a complete and sharp description of $L^p-L^q$ 
mapping properties of $\mathbb{I}_{\s}^{\ab}$, see also Figures \ref{fig2}-\ref{fig4} below. 
\begin{thm} \label{thm:LpLqJacL}
Let $\ab > -1$, $\alpha+\beta \neq -1$, $\s>0$ and $1\le p,q \le \infty$. Set
$$
\kappa:=\Big(\alpha+\frac{1}{2}\Big) \wedge \Big(\beta + \frac{1}2\Big).
$$
Then $\mathbb{I}_{\s}^{\ab}$ has the following mapping properties with respect to the measure space
$((0,\pi),d\t)$.
\begin{itemize}
\item[(a)] Assume that $\kappa \ge 0$.
	\begin{itemize}
		\item[(a1)] If $\s > 1/2$, then $\mathbb{I}_{\s}^{\ab}$ is of strong type $(p,q)$
			for all $p$ and $q$.
		\item[(a2)] If $\s=1/2$, 
				then $\mathbb{I}_{\s}^{\ab}$ is of strong type
			$(p,q)\neq (1,\infty)$.
		\item[(a3)] If $\s < 1/2$, then $\mathbb{I}_{\s}^{\ab}$ is of strong type $(p,q)$ provided that
			$$
				\frac{1}{q} \ge \frac{1}p - 2\s \quad \textrm{and} \quad (p,q) \notin
				\Big\{\Big(1,\frac{1}{1-2\s}\Big), \Big(\frac{1}{2\s},\infty\Big)\Big\}.
			$$
			Moreover, $\mathbb{I}_{\s}^{\ab}$ is of weak type $(1,\frac{1}{1-2\s})$ and of restricted weak type
			$(\frac{1}{2\s},\infty)$.
	\end{itemize}
\item[(b)] Assume now that $\kappa < 0$.
	\begin{itemize}
		\item[(b1)] If $\s > \kappa+1/2$, then $\mathbb{I}_{\s}^{\ab}$ is of strong type $(p,q)$ provided that
			$\frac{1}p < 1+\kappa$ and $\frac{1}q > -\kappa$. Furthermore, $\mathbb{I}_{\s}^{\ab}$ is of weak
			type $(p,\frac{1}{-\kappa})$ for $\frac{1}p < {1+\kappa}$ and of restricted weak type
			$(\frac{1}{1+\kappa},q)$ for $\frac{1}q \ge {-\kappa}$.
		\item[(b2)] If $\s=\kappa + 1/2$, then $\mathbb{I}_{\s}^{\ab}$ has the mapping properties from (b1),
			except for that it is not of restricted weak type $(\frac{1}{1+\kappa},\frac{1}{-\kappa})$.
		\item[(b3)] If $\s < \kappa +1/2$, then $\mathbb{I}_{\s}^{\ab}$ is of strong type $(p,q)$ when
			$$
				\frac{1}p < 1+\kappa \quad \textrm{and} \quad \frac{1}q > -\kappa 
				\quad \textrm{and} \quad \frac{1}q \ge \frac{1}p -2\s.
			$$
			Further, $\mathbb{I}_{\s}^{\ab}$ is of weak type $(p,\frac{1}{-\kappa})$ for 
			$\frac{1}p < {2\s-\kappa}$ and of restricted weak type 
			$(\frac{1}{2\s-\kappa},\frac{1}{-\kappa})$ and
			$(\frac{1}{1+\kappa},q)$ for $\frac{1}q \ge {1+\kappa-2\s}$.
	\end{itemize}
\end{itemize}
All the results in parts (a) and (b) are sharp in the sense that for no pair $(p,q)$ weak type can be
replaced by strong type, and similarly if restricted weak type $(p,q)$ is claimed, then for no 
such $(p,q)$ it can be replaced by weak type. For $(p,q)$ not covered by (a) and (b), 
$\mathbb{I}_{\s}^{\ab}$ is not of restricted weak type $(p,q)$.
\end{thm}
\begin{figure}
\includegraphics[height=0.5\textwidth]{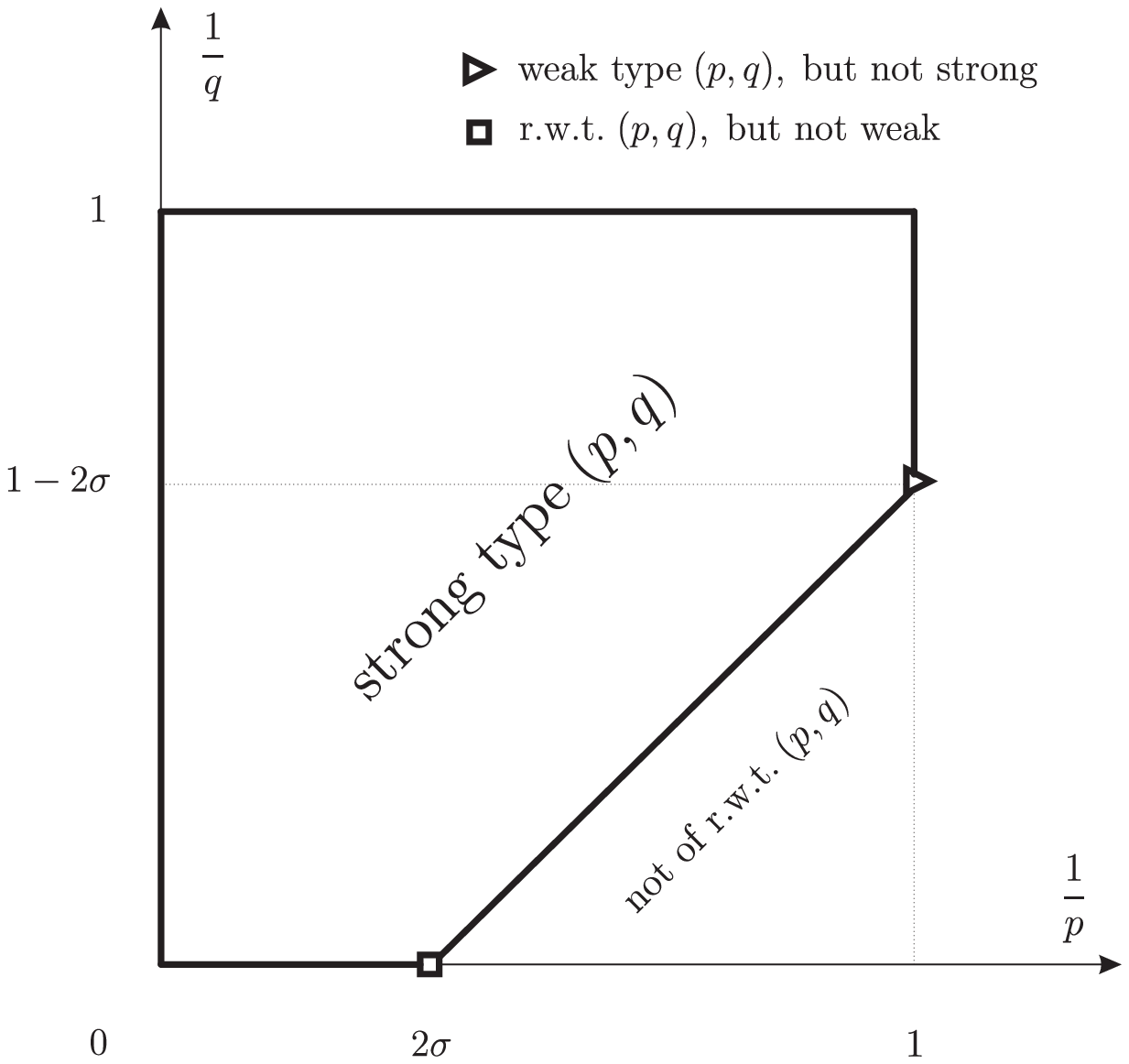}
\caption{Mapping properties of $\mathbb{I}_{\s}^{\ab}$ for $\ab \ge -1/2$
	and $0 < \s < 1/2$.}
 \label{fig2}
\end{figure}
\begin{figure}[ht]
\includegraphics[height=0.5\textwidth]{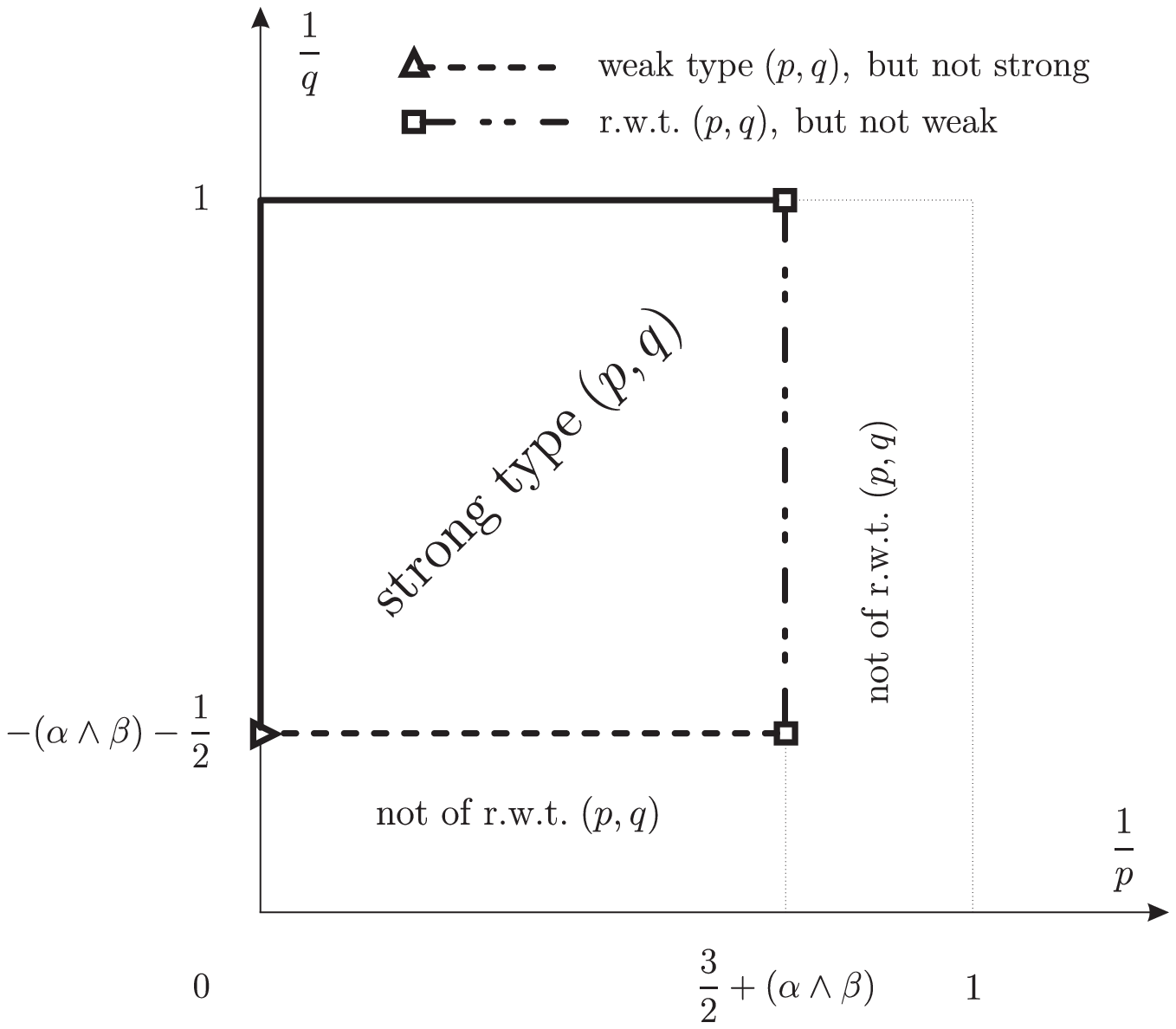}
\caption{Mapping properties of $\mathbb{I}_{\s}^{\ab}$ when $\alpha \wedge \beta < -1/2$
	and $\s > (\alpha \wedge \beta) +1$.}
 \label{fig3}
\end{figure}
\begin{figure}
\includegraphics[height=0.5\textwidth]{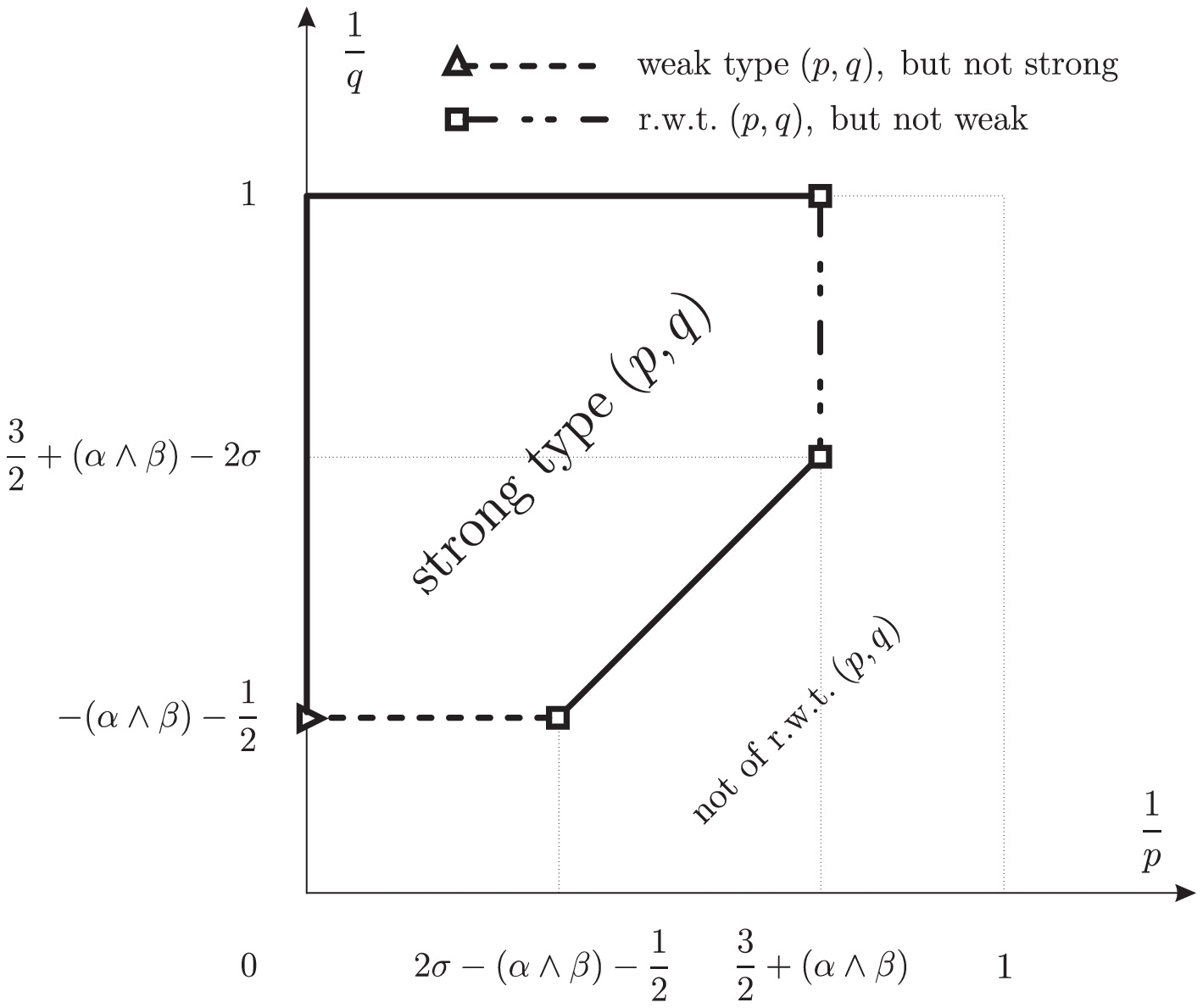}
\caption{Mapping properties of $\mathbb{I}_{\s}^{\ab}$ when $\alpha \wedge \beta < -1/2$
	and $\s < (\alpha \wedge \beta) +1$.}
 \label{fig4}
\end{figure}

\subsection{Natural measure Fourier-Bessel setting}
Let $J_{\nu}$ denote the Bessel function of the first kind and order $\nu > -1$, and let
$\{s_{n,\nu} : n \ge 1\}$ be the sequence of successive positive zeros of $J_{\nu}$, see \cite{Wat}
for the related theory. For $\nu>-1$, define
$$
\phi_n^{\nu}(x) = c_n^{\nu}\, x^{-\nu} J_{\nu}(s_{n,\nu}x),
	\qquad n \ge 1, \quad x \in (0,1),
$$
where $c_n^{\nu}$ are normalizing constants. The Fourier-Bessel system
$\{\phi_n^{\nu} : n \ge 1\}$ is an orthonormal basis in $L^2(d\mu_{\nu})$, where
$d\mu_{\nu}$ is a power density measure in the interval $(0,1)$ given by
$$
d\mu_{\nu}(x) = x^{2\nu+1} \, dx.
$$
Each $\phi_n^{\nu}$ is an eigenfunction of the Bessel operator
$$
\mathcal{L}^{\nu} = - \frac{d^2}{dx^2} - \frac{2\nu+1}x \, \frac{d}{dx},
$$
and we have
$$
\mathcal{L}^{\nu} \phi_n^{\nu} = s_{n,\nu}^2\, \phi_n^{\nu}, \qquad n \ge 1.
$$
We denote by the same symbol $\mathcal{L}^{\nu}$ the natural self-adjoint extension of
the Bessel operator in this context, see \cite{NR1,NR2}.

The integral kernel $\mathcal{H}_t^{\nu}(x,y)$ of the semigroup $\{\exp(-t(\mathcal{L}^{\nu})^{1/2})\}$
was investigated in \cite{NR1}. In particular, in \cite[Theorem 3.4]{NR1} sharp estimates of
$\mathcal{H}_t^{\nu}(x,y)$ were established, but only for a discrete set of half-integer parameters 
$\nu$. Here we prove that these sharp bounds hold in fact for all $\nu>-1$.
\begin{thm} \label{thm:FBPest}
Let $\nu > -1$. Given any $T>0$, we have
\begin{equation*}
\mathcal{H}^{\nu}_t(x,y)  \simeq (1-x)(1-y) \begin{cases}
	( t+ x+y )^{-2\nu-1}
	( t + 2-x -y )^{-2} \frac{t}{t^2+(x-y)^2}, & t \le T \\
	\exp\left( -t s_{1,\nu} \right), & t > T
	\end{cases},
\end{equation*}
uniformly in $t>0$ and $x,y \in (0,1)$.
\end{thm}

Assume that $\s>0$ and consider the potential operator
$$
\mathcal{I}^{\nu}_{\s}f(x) = \int_0^{1} \mathcal{K}_{\s}^{\nu}(x,y)f(y)\, d\mu_{\nu}(y),
$$
where 
\begin{equation*}
\mathcal{K}_{\s}^{\nu}(x,y)=\frac{1}{\Gamma(2\s)}\int_0^{\infty}\mathcal{H}_t^{\nu}(x,y) t^{2\s-1}\,dt, 
	\qquad x,y \in (0,1).
\end{equation*}
Theorem \ref{thm:FBPest} allows us to prove the following sharp bounds for this potential kernel.
\begin{thm} \label{thm:potkerFB}
Let $\nu > -1$ and $\s > 0$ be fixed. The estimate
\begin{align*}
\frac{\mathcal{K}_{\sigma}^{\nu}(x,y)}{(1-x)(1-y)} & \simeq 
	1 + \chi_{\{\s = \nu+1\}} \log \frac{2}{x+y}
	+ \chi_{\{\s=3/2\}} \log\frac{2}{2 - x-y} \\
	&\quad  + (x+y)^{2\s-2(\nu+1)}(2-x-y)^{2\s-3}\begin{cases}
   1, & \s>1/2\\
   \log\frac{(x+y)(2-x-y)}{|x-y|}, &\s=1/2\\
    \left(\frac{(x+y)(2-x-y)}{|x-y|}\right)^{1-2\s}, &\s<1/2
    \end{cases}
\end{align*}
holds uniformly in $x,y \in (0,1)$.
\end{thm}

From Theorem \ref{thm:potkerFB} it can be seen that
$L^p(d\mu_{\nu})\subset \domain\mathcal{I}_{\s}^{\nu}$ for $1\le p \le \infty$.
Moreover, Theorem \ref{thm:potkerFB} makes it possible to describe, in a sharp way, $L^p-L^q$
mapping properties of $\mathcal{I}_{\s}^{\nu}$. These turn out to be the same as for the Jacobi
potential operator $\mathcal{I}_{\s}^{\ab}$ with $\alpha=\nu$ and $\beta=-1/2$, 
as stated in the result below; see also Figure \ref{fig1}. The value of $\beta$ here is perhaps a bit
unexpected, since a fundamental connection between the Jacobi and Fourier-Bessel settings involves 
$\beta=1/2$; see \cite{NR2} and Section \ref{ssec}.
\begin{thm} \label{thm:LpLqFB}
Let $\nu > -1$, $\s>0$ and $1\le p,q \le \infty$. Set 
$$
\eta:= (\nu+1) \vee (1/2).
$$
Then $\mathcal{I}_{\s}^{\nu}$ has the following mapping properties with respect to the measure space
$((0,1),d\mu_{\nu})$.
\begin{itemize}
\item[(i)] If $\s > \eta$, then $\mathcal{I}_{\s}^{\nu}$ is of strong type $(p,q)$ for all $p$ and $q$.
\item[(ii)] If $\s = \eta$, then $\mathcal{I}_{\s}^{\nu}$ is of strong type $(p,q)$ 
	for $(p,q)\neq (1,\infty)$, and not of restricted weak type $(1,\infty)$.
\item[(iii)] Assume finally that $\s < \eta$.
	Then $\mathcal{I}_{\s}^{\nu}$ is of strong type $(p,q)$ provided that
	$$
		\frac{1}q \ge \frac{1}p -\frac{\s}{\eta} \quad \textrm{and} \quad (p,q) \notin
			\Big\{\Big(1,\frac{\eta}{\eta-\s}\Big), \Big(\frac{\eta}{\s},\infty\Big)\Big\}.
	$$
	Moreover, $\mathcal{I}_{\s}^{\nu}$ is of weak type $\big(1,\frac{\eta}{\eta-\s}\big)$ and of restricted
	weak type $\big(\frac{\eta}{\s},\infty\big)$.
	
	These results are sharp in the sense that $\mathcal{I}_{\s}^{\nu}$ is not of strong type
  $\big(1,\frac{\eta}{\eta-\s}\big)$, not of weak type $\big(\frac{\eta}{\s},\infty\big)$, and
  not of restricted weak type $(p,q)$ when $\frac{1}q < \frac{1}p - \frac{\s}{\eta}$.
\end{itemize}
\end{thm}

\subsection{Lebesgue measure Fourier-Bessel setting}
This context emerges from incorporating the measure $\mu_{\nu}$ into the system $\{\phi_n^{\nu}\}$,
see for instance \cite{NR1,NR2}. In this way we derive the Fourier-Bessel system 
$\{\psi_n^{\nu}: n \ge 1\}$,
$$
\psi_n^{\nu}(x) = x^{\nu+1/2} \phi_n^{\nu}(x), \qquad n \ge 1, \quad x \in (0,1),
$$
which for each $\nu>-1$ is an orthonormal basis in $L^2(dx)$; here $dx$ stands for Lebesgue measure
in the interval $(0,1)$. This system consists of eigenfunctions of the differential operator 
$$
\mathbb{L}^{\nu} = - \frac{d^2}{dx^2} - \frac{1/4-\nu^2}{x^2},
$$
and we have 
$$
\mathbb{L}^{\nu} \psi_n^{\nu} = s_{n,\nu}^2\, \psi_n^{\nu}, \qquad n \ge 1.
$$

The associated Poisson semigroup $\{\exp(-t(\mathbb{L}^{\nu})^{1/2})\}$, generated by means of the
square root of the natural self-adjoint extension of $\mathbb{L}^{\nu}$ in this context, 
has an integral kernel given by (see \cite{NR1,NR2})
\begin{equation} \label{PPFBrel}
\mathbb{H}_t^{\nu}(x,y) = (xy)^{\nu+1/2}\, \mathcal{H}_t^{\nu}(x,y).
\end{equation}
Thus Theorem \ref{thm:FBPest} provides also sharp estimates for $\mathbb{H}_t^{\nu}(x,y)$.

Let $\s>0$ and consider the potential operator
$$
\mathbb{I}_{\s}^{\nu}f(x) = \int_0^1 \mathbb{K}_{\s}^{\nu}(x,y) f(y)\, dy,
$$
where
\begin{equation} \label{KlebFB}
\mathbb{K}_{\s}^{\nu}(x,y) = (xy)^{\nu+1/2} \mathcal{K}_{\s}^{\nu}(x,y).
\end{equation}
Sharp estimates of $\mathbb{K}_{\s}^{\nu}(x,y)$ follow readily from \eqref{KlebFB} and
Theorem \ref{thm:potkerFB}. In particular, in case $\nu \ge -1/2$ one concludes that 
$\domain \mathbb{I}_{\s}^{\nu}$ contains all $L^p(dx)$ spaces, $1\le p \le \infty$. If $\nu<-1/2$, then 
$L^p(dx) \subset \domain \mathbb{I}_{\s}^{\nu}$ provided that $\frac{1}p < \nu + \frac{3}2$.

Another consequence of the bounds for $\mathbb{K}_{\s}^{\nu}(x,y)$ is the result below describing
$L^p-L^q$ mapping properties of $\mathbb{I}_{\s}^{\nu}$.
They occur to be the same as for the Jacobi potential operator $\mathbb{I}_{\s}^{\ab}$ with
$\alpha=\nu$ and $\beta=1/2$, see also Figures \ref{fig2}-\ref{fig4}.
\begin{thm} \label{thm:LpLqFBL}
Let $\nu > -1$, $\s>0$ and $1\le p,q \le \infty$. 
Then $\mathbb{I}_{\s}^{\nu}$ has the following mapping properties with respect to the measure space
$((0,1),dx)$.
\begin{itemize}
\item[(a)] Assume that $\nu \ge -1/2$.
	\begin{itemize}
		\item[(a1)] If $\s > 1/2$, then $\mathbb{I}_{\s}^{\nu}$ is of strong type $(p,q)$
			for all $p$ and $q$.
		\item[(a2)] If $\s=1/2$, 
				then $\mathbb{I}_{\s}^{\nu}$ is of strong type
			$(p,q)\neq (1,\infty)$. 
		\item[(a3)] If $\s < 1/2$, then $\mathbb{I}_{\s}^{\nu}$ is of strong type $(p,q)$ provided that
			$$
				\frac{1}{q} \ge \frac{1}p - 2\s \quad \textrm{and} \quad (p,q) \notin
				\Big\{\Big(1,\frac{1}{1-2\s}\Big), \Big(\frac{1}{2\s},\infty\Big)\Big\}.
			$$
			Moreover, $\mathbb{I}_{\s}^{\nu}$ is of weak type $(1,\frac{1}{1-2\s})$ and of restricted weak type
			$(\frac{1}{2\s},\infty)$.
	\end{itemize}
\item[(b)] Assume now that $\nu < -1/2$.
	\begin{itemize}
		\item[(b1)] If $\s > \nu +1$, then $\mathbb{I}_{\s}^{\nu}$ is of strong type $(p,q)$ provided that
			$\frac{1}p < \nu + \frac{3}2$ and $\frac{1}q > -\nu - \frac{1}2$. 
			Furthermore, $\mathbb{I}_{\s}^{\nu}$ is of weak
			type $(p,\frac{-1}{\nu+1/2})$ for $\frac{1}p < {\nu+\frac{3}2}$ and of restricted weak type
			$(\frac{1}{\nu+3/2},q)$ for $\frac{1}q \ge {-\nu-\frac{1}2}$.
		\item[(b2)] If $\s=\nu+1$, then $\mathbb{I}_{\s}^{\nu}$ has the mapping properties from (b1),
			except for that it is not of restricted weak type $(\frac{1}{\nu+3/2},\frac{-1}{\nu+1/2})$.
		\item[(b3)] If $\s < \nu +1$, then $\mathbb{I}_{\s}^{\nu}$ is of strong type $(p,q)$ when
			$$
				\frac{1}p < \nu +\frac{3}2 \quad \textrm{and} \quad \frac{1}q > -\nu-\frac{1}2 \quad \textrm{and}
					\quad \frac{1}q \ge \frac{1}p -2\s.
			$$
			Further, $\mathbb{I}_{\s}^{\nu}$ is of weak type $(p,\frac{-1}{\nu+1/2})$ for 
			$\frac{1}p < {2\s-\nu -\frac{1}2}$ and of restricted weak type 
			$(\frac{1}{2\s-\nu-1/2},\frac{-1}{\nu+1/2})$ and
			$(\frac{1}{\nu+3/2},q)$ for $\frac{1}q \ge {\nu+\frac{3}2-2\s}$.
	\end{itemize}
\end{itemize}
All the results in parts (a) and (b) are sharp in the sense described in Theorem \ref{thm:LpLqJacL}.
\end{thm}

\section{Estimates of the potential kernels} \label{sec:ker}

In this section we prove Theorems \ref{thm:potker}, \ref{thm:FBPest} and \ref{thm:potkerFB}.
We begin with an auxiliary technical result that gives sharp description of the behavior of the integral
$$
J_{\gamma}(T,S,w):=\int_{T}^S\frac{t^\gamma \, dt}{t^2+w^2}
$$
considered as a function of $T,S$ and $w$.
\begin{lem}\label{lem:intJ}
Let $\gamma\in \R$ and $M>0$ be fixed.
\begin{itemize}
\item[(a)] If $\gamma > -1$ then
$$
J_{\gamma}(T,S,w) \simeq \frac{S-T}{S} \, \frac{S^{\gamma+1}}{(S \vee w)^2} \begin{cases}
	1, & \gamma > 1 \\
	1+\log^+ \frac{S}{T \vee {w}}, & \gamma =1\\
	\left( \frac{T\vee {w}}{S\vee {w}}\right)^{\gamma-1}, & \gamma \in (-1,1)
	\end{cases}
$$
uniformly in $0 \le T \le S < \infty$ and $0 < w \le M$.
\item[(b)] If $\gamma \le -1$ then
$$
J_{\gamma}(T,S,w) \simeq \frac{S-T}{S} \, \frac{T^{\gamma+1}}{(T\vee w)^2} \begin{cases}
		1+\log^+ \frac{S\wedge {w}}{T}, & \gamma =-1 \\
		1, & \gamma < -1
		\end{cases} 
$$
uniformly in $0 < T \le S < \infty$ and $0 < w \le M$.
\end{itemize}
\end{lem}
In our applications of Lemma \ref{lem:intJ} below, we will always have $w < S$. However, we decided
to state the result in a slightly more general form than actually needed. This allows one to see the
symmetry between items (a) and (b), and also to understand better the behavior of $J_{\gamma}(T,S,w)$.
The latter may be of independent interest.

In the proof of Lemma \ref{lem:intJ} we will use the following simple estimate.
\begin{lem}\label{lem:diff}
Given $\xi \in \R$, $\xi \neq 0$, we have
$$
|A^{\xi}-B^{\xi}|\simeq |A-B| \begin{cases}
	(A\vee B)^{\xi-1}, & \xi >0 \\
	(A \wedge B)^{\xi+1} /(AB), & \xi <0
	\end{cases}
$$
uniformly in $A,B>0$.
\end{lem}
\begin{proof}
For $\xi > 0$ the estimate is elementary. The case $\xi < 0$ follows then from the result for positive
$\xi$ by replacing $A$ and $B$ by their inverses, respectively, and $\xi$ by its opposite.
\end{proof}

\begin{proof}[Proof of Lemma \ref{lem:intJ}]
The case $T=S$ is trivial, so let $T<S$.
The change of variable $t \to {M}t$ shows that we may assume that $M=1$.
Moreover, we can always consider $T>0$ since then the result asserted for $T=0$ in item (a) follows
by a limiting argument.
In the next step we will further reduce the proof to the case $S=1$.

Consider first $S \le {w}$. Then
$$
J_{\gamma}(T,S,w) \simeq \frac{1}{w^2} \int_T^S t^{\gamma}\, dt \simeq \frac{1}{w^2} \begin{cases}
	|S^{\gamma+1}-T^{\gamma+1}|, & \gamma \neq -1 \\
	\log(S/T), & \gamma = -1
	\end{cases}
$$
and applying Lemma \ref{lem:diff} we see that
$$
J_{\gamma}(T,S,w) \simeq \frac{1}{w^2} \begin{cases}
	(S-T) S^{\gamma}, & \gamma > -1 \\
	\log(S/T), & \gamma = -1 \\
	(S-T) T^{\gamma +1}/S, & \gamma < -1
	\end{cases}.
$$
As easily verified, this coincides with the asserted bounds when $\gamma \neq -1$.
The same is true for $\gamma =-1$, but this case requires perhaps some comment. Namely, the relevant
relation
$$
\log\frac{S}{T} \simeq \frac{S-T}S \Big( 1+\log \frac{S}{T}\Big)
$$
can be checked by distinguishing the cases $2T< S$ and $S/2 \le T < S$ and using in addition
(in the latter case) the bounds
\begin{equation} \label{log_bd}
\log x \simeq x-1, \qquad 1 \le x \le C,
\end{equation}
where $C< \infty$ is a fixed constant.

Next, we consider the complementary range $S> {w}$.
Changing the variable of integration $t \to S t$, we obtain
$$
J_{\gamma}(T,S,w) = S^{\gamma-1}\int_{T/S}^1 \frac{t^{\gamma}\, dt}{t^2+w^2/S^2}
	= S^{\gamma-1} J_{\gamma}(T/S,1,w/S),
$$
and here $w/S<1$. Assuming that the bounds of Lemma \ref{lem:intJ} are true for $S=1$ and applying
them to the last expression we get their validity for general $S$.

Summing up, the proof will be finished once we estimate suitably the integral
$$
J_{\gamma}(T,1,w) = \int_T^1 \frac{t^{\gamma}\, dt}{t^2+w^2}.
$$
The case when $T \ge {w}$ is straightforward. We have
$$
J_{\gamma}(T,1,w) \simeq \int_T^1 t^{\gamma -2}\, dt \simeq \begin{cases}
	|1-T^{\gamma-1}|, & \gamma \neq 1\\
	\log(1/T), & \gamma = 1
	\end{cases}.
$$
Applying now Lemma \ref{lem:diff} and using in addition \eqref{log_bd} when $\gamma =1$, we get
$$
J_{\gamma}(T,1,w) \simeq (1-T) \begin{cases}
	1, & \gamma > 1\\
	1+\log(1/T), & \gamma =1 \\
	T^{\gamma-1}, & \gamma < 1
	\end{cases},
$$
as needed. It remains to treat the case $T < {w}$. Observe that
$$
J_{\gamma}(T,1,w) \simeq \frac{1}{w^2} \int_T^{{w}} t^{\gamma}\, dt + \int_{{w}}^1 t^{\gamma-2}\, dt
	\equiv \mathcal{I}_1 + \mathcal{I}_2.
$$
We must show that
\begin{equation*} 
\mathcal{I}_1 + \mathcal{I}_2 \simeq (1-T) \begin{cases}
	1, & \gamma > 1\\
	1+\log(1/{w}), & \gamma =1\\
	{w}^{\gamma-1}, & -1<\gamma<1\\
	\frac{1}{w^2}\left(1+\log \frac{{w}}{T}\right), & \gamma =-1\\
	T^{\gamma+1}/w^2, & \gamma < -1
	\end{cases}
\end{equation*}
uniformly in $0 < T < {w}\le 1$. To this end we always assume that $T < {w}$.

With the aid of Lemma \ref{lem:diff} one easily finds that
$$
\mathcal{I}_1 \simeq \frac{1}{w^2}\begin{cases}
	({w}-T){w}^{\gamma}, & \gamma > -1 \\
	\log({w}/T), & \gamma = -1\\
	({w}-T) T^{\gamma+1}/{w}, & \gamma < -1
	\end{cases},
\qquad
\mathcal{I}_2 \simeq \begin{cases}
	1-{w}, & \gamma > 1 \\
	\log(1/{w}), & \gamma =1 \\
	(1-{w}){w}^{\gamma-1}, & \gamma < 1
	\end{cases}.
$$
We will analyze separately each of the five cases emerging naturally from the ranges of $\gamma$
appearing above. This will finish the proof.\\
\textbf{Case 1: $\gamma > 1$.} We have
$$
\mathcal{I}_1 + \mathcal{I}_2 \simeq {w}^{\gamma-2}({w}-T) + 1 -{w}.
$$
Clearly, if $1/2\le {w} \le 1$ then $\mathcal{I}_1 + \mathcal{I}_2 \simeq 1-T$.
On the other hand, if ${w}<1/2$ then $\mathcal{I}_1 + \mathcal{I}_2\simeq 1 \simeq 1-T$.
The conclusion follows.\\
\textbf{Case 2: $\gamma = 1$.} Now
$$
\mathcal{I}_1 + \mathcal{I}_2 \simeq {w}^{-1}({w}-T) + \log(1/{w}).
$$
If ${w} \ge 1/2$ then, by \eqref{log_bd},
$
\mathcal{I}_1 + \mathcal{I}_2 \simeq 
{w}^{-1} ({w}-T+1-{w})
\simeq 1-T \simeq (1-T)(1+\log 1/{w}).
$
For ${w}< 1/2$ we have
$
\mathcal{I}_1 + \mathcal{I}_2 \simeq 1-T/{w} + \log 1/{w} \simeq \log 1/{w}
	\simeq (1-T)(1+\log 1/{w}).
$
\\
\textbf{Case 3: $-1< \gamma < 1$.} This time
$$
\mathcal{I}_1 + \mathcal{I}_2 \simeq {w}^{\gamma-2}({w}-T) + {w}^{\,\gamma-1}(1-{w})
$$
and so for ${w}\ge 1/2$ we can write
$
\mathcal{I}_1 + \mathcal{I}_2 \simeq {w}-T+1-{w} \simeq (1-T){w}^{\gamma-1},
$
and when ${w} < 1/2$ we have
$
\mathcal{I}_1 + \mathcal{I}_2 \simeq {w}^{\gamma-1} (1-T/{w}) + {w}^{\gamma-1}
	\simeq (1-T){w}^{\gamma-1},
$
as desired.\\
\textbf{Case 4: $\gamma = -1$.} In this case
$$
\mathcal{I}_1 + \mathcal{I}_2 \simeq w^{-2} \log({w}/T) + {w}^{-2} (1-{w}).
$$
If $T \ge 1/2$ then by \eqref{log_bd} we see that
$
\mathcal{I}_1 + \mathcal{I}_2 \simeq \log {w}/T + 1-{w} \simeq {w}/T-1 +1-{w}
	\simeq 1-T \simeq (1-T) w^{-2} (1+\log {w}/T).
$
When $T<1/2$ we write
$
\mathcal{I}_1 + \mathcal{I}_2 \simeq w^{-2} (1-{w}+\log {w}/T) \simeq
	(1-T) w^{-2} (1+\log{w}/T),
$
where the last relation is verified by considering separately the subcases ${w} < 3/4$
and ${w}>3/4$.\\
\textbf{Case 5: $\gamma < -1$.} We now have
$$
\mathcal{I}_1 + \mathcal{I}_2 \simeq \frac{{w}-T}{{w}} \,\frac{T^{\gamma+1}}{w^2}
	+ \frac{{w}^{\gamma+1}}{w^2} (1-{w}).
$$
For $T\ge 1/2$ it follows that
$
\mathcal{I}_1 + \mathcal{I}_2 \simeq w^{-2} T^{\gamma+1} ({w}-T) + w^{-2}T^{\gamma+1}(1-{w})
	\simeq (1-T)w^{-2} T^{\gamma+1}.
$
When $T < 1/2$ the same estimates are justified by considering separately the subcases
$T \ge 3{w}/4$ and $T < 3{w}/4$. More precisely, in the first subcase
$\mathcal{I}_1 \lesssim T^{\gamma+1}/w^2 \simeq \mathcal{I}_2$ and in the second one
$\mathcal{I}_2 \lesssim {w}^{\gamma+1}/w^2 < T^{\gamma+1}/w^2 \simeq \mathcal{I}_1$.
The conclusion follows.
\end{proof}

\subsection{Estimates of the Jacobi potential kernels}
With Theorem \ref{thm:estPJ} and Lemma \ref{lem:intJ} at our disposal, we can now verify the sharp estimate
of $\mathcal{K}_{\s}^{\ab}(\t,\v)$ stated in Theorem \ref{thm:potker}.
\begin{proof}[Proof of Theorem \ref{thm:potker}]
We may assume that
\begin{equation} \label{ass}
\t + \v \le 2\pi -\t -\v
\end{equation}
since then the complementary case follows by replacing $\t$ and $\v$ by $\pi-\t$ and $\pi-\v$, respectively,
and exchanging the roles of $\a$ and $\b$. Note that \eqref{ass} is equivalent to each of the inequalities
$$
\t+\v \le \pi, \qquad 2\pi - \t-\v \ge \pi.
$$
In particular, \eqref{ass} implies 
$$
2\pi-\t-\v \simeq 1, \qquad \t,\v \in (0,\pi).
$$
These relations will be used throughout the proof without further mention.

In view of \eqref{ass}, the estimates we must show read as
\begin{equation}
\mathcal{K}_{\s}^{\ab}(\t,\v) \simeq 1+ \chi_{\{\s=\a+1\}} 
	\log\frac{2\pi}{\t+\v} + (\t+\v)^{-2\a-1} \begin{cases}
	(\t+\v)^{2\s-1}, & \s > 1/2 \\
	1+ \log\frac{\t+\v}{|\t-\v|}, & \s=1/2 \\
	|\t-\v|^{2\s-1}, & \s < 1/2
	\end{cases}. \label{aim} 
\end{equation}
Using Theorem \ref{thm:estPJ} we can write, uniformly in $\t,\v \in (0,\pi)$,
\begin{align*}
\mathcal{K}_{\s}^{\ab}(\t,\v) \simeq & \;(\t+\v)^{-2\a-1}\int_0^{\t+\v} \frac{t^{2\s}\, dt}{t^2+(\t-\v)^2}
	+ \int_{\t+\v}^{2\pi-\t-\v} \frac{t^{2\s-2\a-1}\,dt}{t^2+(\t-\v)^2} \\
	& \quad+ \int_{2\pi-\t-\v}^{2\pi} \frac{t^{2\s-2\a-2\b-2}\,dt}{t^2+(\t-\v)^2}
	+ \int_{2\pi}^{\infty} \exp\left( -t \frac{|\a+\b+1|}2\right) t^{2\s-1}\, dt \\
\equiv & \;\mathcal{J}_1 + \mathcal{J}_2 + \mathcal{J}_3 + \mathcal{J}_4.
\end{align*}
Clearly, all the components here are nonnegative and $\mathcal{J}_4 = c(\s,\ab) \simeq 1$.
Moreover, by \eqref{ass},
$$
\mathcal{J}_3 \lesssim \int_{\pi}^{2\pi} \frac{t^{2\s-2\a-2\b-2}\, dt}{t^2+(\t-\v)^2}
	\lesssim \int_{\pi}^{2\pi}  t^{2\s-2\a-2\b-4} \, dt \simeq 1.
$$
Therefore
$$
\mathcal{J}_3 + \mathcal{J}_4 \simeq 1.
$$

To describe the behavior of $\mathcal{J}_1$ and $\mathcal{J}_2$ we apply Lemma \ref{lem:intJ}.
More precisely, using Lemma \ref{lem:intJ} (a) with $M=2\pi$, $\gamma = 2\s$, $T=0$,
$S=\t+\v$ and $w=|\t-\v|$, we get
$$
\mathcal{J}_1 \simeq (\t+\v)^{-2\a-1} \begin{cases}
	(\t+\v)^{2\s-1}, & \s > 1/2 \\
	1+ \log\frac{\t+\v}{|\t-\v|}, & \s=1/2 \\
	|\t-\v|^{2\s-1}, & \s < 1/2
	\end{cases}.
$$
Letting $\gamma = 2\s-2\a-1$, $T=\t+\v$, $S=2\pi-\t-\v$ and applying again Lemma \ref{lem:intJ}
(item (a) when $\s>\a$ and item (b) for $\s \le \a$) leads to the bound
$$
\mathcal{J}_2 \simeq (\pi-\t-\v) \begin{cases}
	1, & \s > \a+1 \\
	1+\log\frac{2\pi}{\t+\v}, & \s=\a+1 \\
	(\t+\v)^{2\s-2\a-2}, & \s < \a+1
	\end{cases}.
$$

Combining these estimates of $\mathcal{J}_1$ and $\mathcal{J}_2$ we see that when $\s \neq \a+1$
\begin{align*}
\mathcal{J}_2 & \lesssim \begin{cases}
	1, & \s > \a +1\\
	(\t+\v)^{-2\a-1} (\t+\v)^{2\s-1}, & \s < \a+1
	\end{cases} \\
	& \lesssim 1 + \mathcal{J}_1,
\end{align*}
and in the singular case $\s=\a+1$ we have
$$
1 + \mathcal{J}_2 \simeq \log\frac{2\pi}{\t+\v}.
$$
Since $\mathcal{J}_1$ is comparable with the third component in \eqref{aim} and
$1\simeq \mathcal{J}_3 + \mathcal{J}_4$ with the first one, the conclusion follows.
\end{proof}

\subsection{Estimates of the Poisson and potential kernels in the Fourier-Bessel settings} \label{ssec}
Our first objective now is to
prove Theorem \ref{thm:FBPest}. To proceed, we consider the Jacobi trigonometric `function'
setting scaled to the interval $(0,1)$, see \cite[Section 2]{NR2}. Define
$$
\widetilde{\phi}_n^{\ab}(x) = \sqrt{\pi} \phi_n^{\ab}(\pi x), \qquad n \ge 0, \quad x \in (0,1),
$$
where $\phi_n^{\ab}$ are as in \eqref{Jtf}.
Then the system $\{\widetilde{\phi}_n^{\ab} : n \ge 0\}$ is an orthonormal basis in $L^2(dx)$,
$dx$ being Lebesgue measure in $(0,1)$. Moreover, each $\widetilde{\phi}_n^{\ab}$ is an eigenfunction
of the differential operator
$$
\widetilde{\mathbb{J}}^{\ab} = -\frac{d^2}{dx^2} - \frac{\pi^2(1/4-\alpha^2)}{4\sin^2(\pi x/2)}
	- \frac{\pi^2(1/4-\beta^2)}{4\cos^2(\pi x/2)}
$$
and we have
$$
\widetilde{\mathbb{J}}^{\ab} \widetilde{\phi}_n^{\ab} = \pi^2 \Big( n + \frac{\alpha+\beta+1}{2}\Big)^2
	\widetilde{\phi}_n^{\ab}, \qquad n \ge 0.
$$

The heat and Poisson kernels in this context are given by 
\begin{align*}
\widetilde{\mathbb{G}}_t^{\ab}(x,y) & = \sum_{n=0}^{\infty} 
	\exp\bigg(-t\pi^2 \Big(n+\frac{\alpha+\beta+1}2\Big)^2\bigg)\, \widetilde{\phi}_n^{\ab}(x)\,
		\widetilde{\phi}_n^{\ab}(y), \\ 
\widetilde{\mathbb{H}}_t^{\ab}(x,y) & = \sum_{n=0}^{\infty} 
	\exp\bigg(-t\pi \Big|n+\frac{\alpha+\beta+1}2\Big|\bigg)\, \widetilde{\phi}_n^{\ab}(x)\,
		\widetilde{\phi}_n^{\ab}(y),		
\end{align*}
respectively. Large time behavior of these kernels can be described by means of the above oscillating 
series. Indeed, taking into account that $\widetilde{\phi}_0^{\ab}(x)$ is a constant times
$(\sin(\pi x/2))^{\alpha+1/2}(\cos(\pi x/2))^{\beta+1/2}$ and that (see e.g.~\cite[(14)]{NSS})
$$
\big|\widetilde{\phi}_n^{\ab}(x)\big| \lesssim \Big(\sin\frac{\pi x}2\Big)^{\alpha+1/2}
	\Big(\cos\frac{\pi x}2 \Big)^{\beta+1/2} n^{\alpha+\beta+2}, \qquad n \ge 1, \quad x \in (0,1),
$$
we conclude that for large $t$ the above series behave like their first terms.
More precisely, in case of the heat kernel we have the following.
\begin{propo} \label{prop:larget}
For $T$ sufficiently large, 
$$
\widetilde{\mathbb{G}}_t^{\ab}(x,y) \simeq (xy)^{\alpha+1/2} \big((1-x)(1-y)\big)^{\beta+1/2}\, 
	\exp\bigg(-t\pi^2 \Big(\frac{\alpha+\beta+1}2\Big)^2\bigg), \qquad t \ge T, 
$$
uniformly in $x,y \in (0,1)$.
\end{propo}
As for the short time behavior of $\widetilde{\mathbb{H}}_t^{\ab}(x,y)$, Theorem \ref{thm:estPJ}
combined with the relation \eqref{PPrel} and a simple scaling argument leads to the estimate
\begin{equation} \label{Htilde}
\widetilde{\mathbb{H}}_t^{\ab}(x,y) \simeq \bigg( \frac{\sqrt{xy}}{t+x+y}\bigg)^{2\alpha+1}
	\bigg( \frac{\sqrt{(1-x)(1-y)}}{t+2-x-y}\bigg)^{2\beta+1} \frac{t}{t^2+(x-y)^2}, \qquad 0 < t \le T,
\end{equation}
uniformly in $x,y \in (0,1)$, where $T>0$ is arbitrary and fixed.

\begin{proof}[{Proof of Theorem \ref{thm:FBPest}}]
Let $T>0$, to be fixed later. Notice that since the short and long time bounds of the theorem
coincide for $t$ staying in a fixed interval $[T_0,T_1]$, with $0<T_0<T_1<\infty$, we may 
prove the result with $T$ chosen as large as we wish.

The estimate of $\mathcal{H}_t^{\nu}(x,y)$ for $t \ge T$ follows from \cite[Theorem 3.7]{NR1}, provided
that $T$ is large enough, and we may assume this is the case. 
Thus it remains to verify the short time estimate. Further, in view of \eqref{PPFBrel}, 
it is sufficient to show that the Poisson kernel $\mathbb{H}_t^{\nu}(x,y)$
in the Lebesgue measure Fourier-Bessel setting has for $t \le T$ the same bounds as 
$\widetilde{\mathbb{H}}_t^{\ab}(x,y)$ in \eqref{Htilde} above, with $\alpha=\nu$ and $\beta=1/2$.

Let $\mathbb{G}_t^{\nu}(x,y)$ be the heat kernel related to the Lebesgue measure Fourier-Bessel context.
By the subordination principle,
$$
\mathbb{H}_t^{\nu}(x,y) = \frac{t}{\sqrt{4\pi}} \int_0^{\infty} \mathbb{G}_{u}^{\nu}(x,y)
	e^{-t^2/(4u)} u^{-3/2}\, du
$$
and similarly for $\widetilde{\mathbb{H}}_t^{\nu,1/2}(x,y)$ and $\widetilde{\mathbb{G}}_t^{\nu,1/2}(x,y)$.
Set
\begin{align*}
J_0 = \int_0^T \mathbb{G}_u^{\nu}(x,y) e^{-t^2/(4u)} u^{-3/2}\, du, & \qquad
J_{\infty} = \int_T^{\infty} \mathbb{G}_u^{\nu}(x,y) e^{-t^2/(4u)} u^{-3/2}\, du, \\
\widetilde{J}_0 = \int_0^T \widetilde{\mathbb{G}}_u^{\nu,1/2}(x,y) e^{-t^2/(4u)} u^{-3/2}\, du, & \qquad
\widetilde{J}_{\infty} = \int_T^{\infty} \widetilde{\mathbb{G}}_u^{\nu,1/2}(x,y) e^{-t^2/(4u)} u^{-3/2}\, du.
\end{align*}
According to \cite[Remark 3.3]{NR2}, we know that
$$
\mathbb{G}_u^{\nu}(x,y) \simeq \widetilde{\mathbb{G}}_u^{\nu,1/2}(x,y), \qquad x,y\in (0,1), \quad 0<u\le T.
$$
Thus $J_0\simeq \widetilde{J}_0$ uniformly in $x,y\in (0,1)$ and $t > 0$.

To show that $J_{\infty} \simeq \widetilde{J}_{\infty}$ uniformly in $x,y\in (0,1)$ and $t \le T$, we
note that by \cite[Theorem 3.7]{NR1} and Proposition \ref{prop:larget}
\begin{align*}
{\mathbb{G}_u^{\nu}(x,y)} &\simeq  {(xy)^{\nu+1/2}(1-x)(1-y)} \exp\big( -u s_{1,\nu}^2\big), 
	\qquad u \ge T,\\
	{\widetilde{\mathbb{G}}_u^{\nu,1/2}(x,y)} & \simeq {(xy)^{\nu+1/2}(1-x)(1-y)}
		\exp\big( -u \pi^2 {(\nu+3/2)^2/4}\big), \qquad u \ge T,
\end{align*}
provided that $T$ is chosen sufficiently large. Now we fix $T$ that is 
simultaneously large enough in all the relevant
places above and observe that the desired comparability of $J_{\infty}$ and $\widetilde{J}_{\infty}$
will follow once we check that, given $c_1,c_2>0$, one has
$$
\int_T^{\infty} e^{-c_1 u} e^{-t^2/(4u)} u^{-3/2}\, du \simeq 
	\int_T^{\infty} e^{-c_2 u} e^{-t^2/(4u)} u^{-3/2}\, du, \qquad 0 < t \le T.
$$
This, however, is clear because $t^2/(4u)$ stays bounded in the integrals above if $t\le T$.

Summing up, we see that
$$
\mathbb{H}_t^{\nu}(x,y) \simeq \widetilde{\mathbb{H}}_t^{\nu,1/2}(x,y), \qquad x,y \in (0,1), \quad t \le T.
$$
This finishes the proof.
\end{proof}

We are now in a position to prove sharp estimates for the Fourier-Bessel potential kernel 
$\mathcal{K}_{\s}^{\nu}(x,y)$.
\begin{proof}[{Proof of Theorem \ref{thm:potkerFB}}]
Observe that by Theorem \ref{thm:FBPest} the expression $\mathcal{H}_t^{\nu}(x,y)/((1-x)(1-y))$
has, up to an obvious scaling, the same short time bounds as the Jacobi-Poisson kernel
$\mathcal{H}_t^{\nu,1/2}(\t,\v)$, see Theorem \ref{thm:estPJ}. Moreover, the long time behaviors are
also the same, up to constants in the arguments of the exponentials. Therefore we can proceed exactly as
in the proof of Theorem \ref{thm:potker} above and conclude the estimate of Theorem~\ref{thm:potkerFB}.
\end{proof}

\section{$L^p-L^q$ estimates} \label{sec:LpLq}

To prove the $L^p-L^q$ estimates stated in Section \ref{sec:prel}, we will need some preparatory
facts and results. A part of them are some basic properties of a generic integral operator
\begin{equation*}
Tf(\t) = \int_{\mathcal{X}} K(\t,\v) f(\v) \, d\mu(\v)
\end{equation*}
related to a measure space $(\mathcal{X},\mu)$. Here, for our purposes,
we fix $\mathcal{X}=(0,\pi)$,  $\mu=\mu_{\ab}$ or $\mu$ be Lebesgue measure,
and we always assume that the kernel $K(\t,\v)$ is nonnegative and symmetric, $K(\t,\v) = K(\v,\t) \ge 0$.
Considering $1 \le p,q \le \infty$, we make the following observations.
\begin{itemize}
\item[(A)] If $T$ is of strong type $(p,q)$, then $T$ is of strong type $(\widetilde{p},\widetilde{q})$
	for $\widetilde{p} \ge p$ and $\widetilde{q} \le q$. Indeed, since $\mu$ is finite, we have
	$L^{r}(d\mu) \subset L^{\widetilde{r}}(d\mu)$ for $r \ge \widetilde{r}$ and the claim follows.
\item[(B)] If $T$ is of strong type $(1,q)$, then $T$ is of strong type $(\widetilde{p},\widetilde{q})$
	provided that $1/\widetilde{q} \ge 1/\widetilde{p} - 1/q'$. Indeed, by duality $T^{*}=T$ is of strong type
	$(q',\infty)$, so the conclusion follows by interpolating between the strong types
		$(1,q)$ and $(q',\infty)$, and (A) above.
\item[(C)] If $T$ is of weak type $(1,q)$, $1 < q < \infty$,
	then $T$ is of restricted weak type $(q',\infty)$
	and of strong type $(\widetilde{p},\widetilde{q})$ for $1/\widetilde{q}=1/\widetilde{p} - 1/q'$,
	$\widetilde{p} > 1$, $\widetilde{q} < \infty$. This is justified as follows. Notice that the weak type
	$(1,q)$ means, in terms of Lorentz spaces, boundedness from $L^1(d\mu)$ to $L^{q,\infty}(d\mu)$. 
	Then the adjoint operator $T^{*}$ maps boundedly $(L^{q,\infty}(d\mu))^{*}$ into 
	$(L^1(d\mu))^{*} = L^{\infty}(d\mu)$. Further, the associate
	space of $L^{q,\infty}(d\mu)$ in the sense of \cite[Chapter 1, Definition 2.3]{BS} is $L^{q',1}(d\mu)$
	(cf. \cite[Chapter 4, Theorem 4.7]{BS}) and by \cite[Chapter 1, Theorem 2.9]{BS} it can be regarded as
	a subspace of the dual of $L^{q,\infty}(d\mu)$. 
	Since $T^{*}=T$, we infer that $T$ is of restricted weak type
	$(q',\infty)$. The remaining assertion follows by an extension of the Marcinkiewicz interpolation 
	theorem for Lorentz spaces due to Stein and Weiss, see \cite[Chapter V, Theorem 3.15]{SW} or
	\cite[Chapter 4, Theorem 5.5]{BS}.
\item[(D)] If $T$ is of weak type $(p,\infty)$, $p < \infty$, then $T$ is of strong type $(1,p')$.
	Actually, by definition, weak type $(p,\infty)$ coincides with strong type $(p,\infty)$, which means
	boundedness from $L^p(d\mu)$ to $L^{\infty}(d\mu)$. 
	Since $T^{*}=T$ and $L^1(d\mu) \subset (L^1(d\mu))^{**} = (L^{\infty}(d\mu))^*$, the conclusion follows.
\end{itemize}

Given $0 < \xi \le 1$ and $f \ge 0$, let 
\begin{equation} \label{Uxi}
U_{\xi}f(\t) = \int_{\mathcal{X}} \frac{|\t-\v|^{\xi}}{\mu(B(\t,|\t-\v|))} f(\v) \, d\mu(\v).
\end{equation}
This operator appears in the literature as a variant of fractional integral related
to spaces of homogeneous type, see \cite[Section 5]{AM} or \cite[Section 1]{K} and references given there.
We shall use the following.
\begin{lem} \label{lem:Ixi}
Let $(\mathcal{X},\mu)$ be as above and fix $0<\xi \le 1$. Assume that there are constants $s>\xi$ and
$c>0$ such that $\mu(B_r) \ge cr^s$ for any ball $B_r$ in $\mathcal{X}$ of radius $r < \diam \mathcal{X}$.
Then the sublinear operator $f \mapsto U_{\xi}|f|$ is bounded from $L^1(d\mu)$ to 
weak $L^{{s}/({s-\xi})}(d\mu)$.
\end{lem}

\begin{proof}
We follow well known arguments going back to Hedberg's paper \cite{Hed}, 
see for instance the proof of \cite[Corollary 5.2]{AM} or the proof
of \cite[Proposition 3.19]{H}. The integral defining $U_{\xi}|f|$ is divided into `good' and `bad'
parts. Then treatment of the good part is straightforward and the bad part is analyzed by means of
a dyadic decomposition, with the aid of the assumed lower estimate for $\mu(B_r)$ and the doubling property
of $\mu$. In this way one arrives at the so-called Hedberg's inequality
$$
U_{\xi}|f|(\t) \lesssim \|f\|_1^{\xi/s} \big( Mf(\t)\big)^{1-\xi/s},
$$
where $M$ stands for the (centered) Hardy-Littlewood maximal function in the space $(\mathcal{X},\mu)$.
Since $M$ satisfies the weak type $(1,1)$ inequality (see \cite[Theorem 2.2]{H}), we get 
$$
\mu\big(\{ U_{\xi}|f| > \lambda \}\big) \lesssim \mu\Big(\Big\{ Mf(\t) > 
	\big({\lambda}/{\|f\|_1^{\xi/s}}\big)^{s/(s-\xi)} \Big\}\Big)
	\lesssim \bigg(\frac{\|f\|_1}{\lambda}\bigg)^{s/(s-\xi)},
$$
uniformly in $\lambda > 0$ and $f \in L^1(d\mu)$.
\end{proof}

\subsection{$\boldsymbol{L^p-L^q}$ estimates in the Jacobi trigonometric polynomial setting} \label{sssec}
Our strategy to prove Theorem \ref{thm:LpLqJac} is based on decomposing
(in the sense of $\simeq$) the potential kernel according to the estimate of Theorem 
\ref{thm:potker}. We write
$$
\mathcal{K}_{\s}^{\ab}(\t,\v) \simeq \sum_{i=1}^6 \mathcal{K}_i(\t,\v),
$$
where
\begin{align*}
\mathcal{K}_1(\t,\v) & := 1, \\
\mathcal{K}_2(\t,\v) & := \chi_{\{\s=\alpha+1\}} \log\frac{2\pi}{\t+\v},\\
\mathcal{K}_3(\t,\v) & := \chi_{\{\s=\beta+1\}} \log\frac{2\pi}{2\pi-\t-\v},\\
\mathcal{K}_4(\t,\v) & := \chi_{\{\s>1/2\}}\, (\t+\v)^{2\s-2(\alpha+1)}(2\pi-\t-\v)^{2\s-2(\beta+1)},\\
\mathcal{K}_5(\t,\v) & := \chi_{\{\s=1/2\}}\, (\t+\v)^{-2\alpha-1}(2\pi-\t-\v)^{-2\beta-1}
	\log\frac{(\t+\v)(2\pi-\t-\v)}{|\t-\v|},\\
\mathcal{K}_6(\t,\v) & := \chi_{\{\s<1/2\}}\, (\t+\v)^{-2\alpha-1}(2\pi-\t-\v)^{-2\beta-1} |\t-\v|^{2\s-1}.
\end{align*}
We denote the corresponding integral operators by $\mathcal{T}_i$, $i=1,\ldots,6$,
$$
\mathcal{T}_i f(\t) = \int_0^{\pi} \mathcal{K}_i(\t,\v) f(\v) \, d\mu_{\ab}(\v).
$$

Notice that all the kernels here are nonnegative and hence, from now on, we may and do assume that $f \ge 0$.
Clearly, to show any of the asserted mapping properties of $\mathcal{I}_{\s}^{\ab}$, it is sufficient to
do the same for each $\mathcal{T}_i$, $i=1,\ldots,6$, separately. On the other hand, to disprove one of the
mapping properties of $\mathcal{I}_{\s}^{\ab}$, it is enough to verify that it fails in case of one
particular $\mathcal{T}_i$.

While studying the proof below, it is convenient to keep in mind Figure \ref{fig1}.

\begin{proof}[{Proof of Theorem \ref{thm:LpLqJac}}]
Let $1\le p,q \le \infty$.
We will show the following mapping properties of $\mathcal{T}_i$, $i=1,\ldots,6$.
As easily seen, altogether they imply all the assertions we need to prove.
Note that the case $\alpha+\beta=-1$ is not excluded below.
\begin{itemize}
\item $\mathcal{T}_1$ is of strong type $(p,q)$ for all $p$ and $q$.
\item $\mathcal{T}_2$, being nontrivial only for $\s=\alpha+1$, is in this case 
	of strong type $(p,q)\neq (1,\infty)$ and not of restricted weak type $(1,\infty)$.
\item $\mathcal{T}_3$, being nontrivial only for $\s=\beta+1$, is in this case 
	of strong type $(p,q)\neq (1,\infty)$ and not of restricted weak type $(1,\infty)$.
\item $\mathcal{T}_4$, being nontrivial only when $\s>1/2$, satisfies in this case the following:
	\begin{itemize}
		\item[$\star$] if $\s \ge \delta$, then $\mathcal{T}_4$ is of strong type $(p,q)$ for all $p$ and $q$;
		\item[$\star$] if $\s < \delta$, then $\mathcal{T}_4$ has the positive and negative mapping
			properties of item (iii) of the theorem.
	\end{itemize}
\item $\mathcal{T}_5$, being nontrivial only for $\s=1/2$, satisfies in this case the following:
	\begin{itemize}
		\item[$\star$] if $\s \ge \delta$ (this actually forces $\s=\delta=1/2$), then $\mathcal{T}_5$ is
			of strong type $(p,q)\neq (1,\infty)$ and not of restricted weak type $(1,\infty)$;
		\item[$\star$] if $\s < \delta$, then $\mathcal{T}_5$ has the positive and negative mapping properties
			of item (iii) of the theorem.
	\end{itemize}
\item $\mathcal{T}_6$, being nontrivial only when $\s < 1/2$ (notice that this implies $\s < \delta$),
	has in this case all the positive and negative mapping properties of item (iii) of the theorem.
\end{itemize}

\noindent \textbf{Analysis of $\boldsymbol{\mathcal{T}_1}$.}
By (A) above, it is enough to verify that $\mathcal{T}_1$ is of strong type $(1,\infty)$, which is trivial.

\noindent \textbf{Analysis of $\boldsymbol{\mathcal{T}_2}$.}
We first show the positive results.
In view of (D) and (B), it is enough to verify the strong type $(p,\infty)$ for $1< p < \infty$.
By H\"older's inequality, we have
\begin{align*}
\|\mathcal{T}_2 f\|_{\infty} & \le \sup_{0<\t<\pi} \int_0^{\pi} \log\frac{2\pi}{\t+\v} f(\v)\, d\mu_{\ab}(\v)
	= \int_0^{\pi} \log\frac{2\pi}{\v} f(\v)\, d\mu_{\ab}(\v)\\
	& \le \|f\|_{p} \bigg( \int_0^{\pi} \Big( \log\frac{2\pi}{\v}\Big)^{p'}\, d\mu_{\ab}(\v)\bigg)^{1/p'}.
\end{align*}
Since the last integral is finite, the conclusion follows. 

To see that $\mathcal{T}_2$ is not of restricted weak type $(1,\infty)$ when $\s=\alpha+1$, we  
let $f_{\varepsilon}=\chi_{(0,\varepsilon)}$
for small $\varepsilon > 0$. Then $\|f_{\varepsilon}\|_{1} \simeq \varepsilon^{2\alpha+2}$ and
$$
\|\mathcal{T}_2 f_{\varepsilon}\|_{\infty} = \essup_{0<\t<\pi} \int_0^{\varepsilon} 
	\log\frac{2\pi}{\t+\v}\, d\mu_{\ab}(\v)
	= \int_0^{\varepsilon} \log\frac{2\pi}{\v}\, d\mu_{\ab}(\v) 
	\simeq \varepsilon^{2\alpha+2} \log\frac{2\pi}{\varepsilon}.
$$
Letting $\varepsilon \to 0$, we infer that the estimate 
$\|\mathcal{T}_2 f_{\varepsilon}\|_{\infty} \lesssim \|f_{\varepsilon}\|_{1}$, 
which is both the strong and weak type $(1,\infty)$ inequality, is not true. 

\noindent \textbf{Analysis of $\boldsymbol{\mathcal{T}_3}$.}
For symmetry reasons, treatment of $\mathcal{T}_3$ is parallel to that of $\mathcal{T}_2$ above.

\noindent \textbf{Analysis of $\boldsymbol{\mathcal{T}_4}$.}
Recall that $\delta = (\alpha+1)\vee (\beta+1) \vee (1/2)$.
Observe that when $\s \ge \delta$ we have $\mathcal{K}_4(\t,\v) \lesssim 1$ and therefore in this case
$\mathcal{T}_4$ shares the positive mapping properties of $\mathcal{T}_1$. 

It remains to analyze the case $1/2<\s<\delta$. To this end, for symmetry reasons, we may and do assume
that $\alpha \ge \beta$. Thus we actually consider the case $1/2<\s<\alpha+1=\delta$.
We will show that $\mathcal{T}_4$ is of weak type $(1,\frac{\delta}{\delta-\sigma})$ 
and of strong type $(1,q)$ for
$1\le q < \frac{\delta}{\delta-\sigma}$. In view of (B) and (C), this will imply that 
$\mathcal{T}_4$ is of strong type $(p,q)$
provided that $\frac{1}{q}\ge \frac{1}{p} - \frac{\s}{\delta}$, except for
$(p,q)= (1,\frac{\delta}{\delta-\sigma})$ and $(p,q) = (\frac{\delta}{\sigma},\infty)$, and in the latter
case $\mathcal{T}_4$ is of restricted weak type. Furthermore, we will prove that $\mathcal{T}_4$ is not of
weak (strong) type $(\frac{\delta}{\sigma},\infty)$ and hence, by (B), 
neither of strong type $(1,\frac{\delta}{\delta-\sigma})$.
Finally, we will check that $\mathcal{T}_4$ is not of restricted weak type $(p,q)$ when 
$\frac{1}{q}< \frac{1}{p}-\frac{\sigma}{\delta}$.

We claim that $\mathcal{T}_4$ is of weak type $(1,\frac{\alpha+1}{\alpha+1-\s})$. We have
\begin{align} \label{est1}
\mathcal{T}_4 f(\t) & \lesssim  \chi_{(0,\pi/2]}(\t) \, \t^{2\s-2(\alpha+1)}  \|f\|_{1} 
	+ \chi_{(\pi/2,\pi)}(\t) \,\big[1\vee(\pi-\t)^{2\s-2(\beta+1)}\big] \|f\|_{1} \\
	& \equiv \mathcal{T}_{4,0} f(\t) + \mathcal{T}_{4,\pi} f(\t). \nonumber
\end{align}
Then, uniformly in $\lambda>0$ and $0 \le f\in L^1(d\mu_{\ab})$,
$$
\mu_{\ab}\big(\{\mathcal{T}_{4,0} f(\t)>\lambda\}\big) 
	\lesssim \int_0^{(\lambda/\|f\|_1)^{1/(2\s-2(\alpha+1))}}
	\t^{2\alpha+1}\, d\t \simeq \bigg(\frac{\|f\|_1}{\lambda}\bigg)^{\frac{\alpha+1}{\alpha+1-\s}},
$$
so $\mathcal{T}_{4,0}$ is of weak type $(1,\frac{\alpha+1}{\alpha+1-\s})$. 
If $\beta=\alpha$, the same argument shows that
also $\mathcal{T}_{4,\pi}$ has this mapping property. 
As easily verified, for $\beta < \alpha$ the operator $\mathcal{T}_{4,\pi}$ is of
strong type $(1,\frac{\alpha+1}{\alpha+1-\s})$. The claim follows.
From the estimate \eqref{est1} it is also clear that $\mathcal{T}_4$ is of strong type $(1,q)$ for 
$q < \frac{\alpha+1}{\alpha+1-\sigma}$.

Passing to the negative results, we first disprove
the weak type $(\frac{\delta}{\sigma},\infty)$ which, by definition, coincides with strong type 
$(\frac{\delta}{\sigma},\infty)$. Recall that $\delta=\alpha+1$.
Take $f(\v) = \chi_{(0,1)}(\v) /(\v^{2\s}\log\frac{2}{\v})$. Then $f \in L^{\delta/\s}(d\mu_{\ab})$ since
$$
\|f\|_{\delta/\s}^{\delta/\s} =
\int_0^{\pi} \big(f(\v)\big)^{\delta/\s}\, d\mu_{\ab}(\v) \simeq 
	\int_0^1 \frac{d\v}{\v (\log\frac{2}{\v})^{\delta/\s}} < \infty.
$$ 
But
$$
\|\mathcal{T}_4 f\|_{\infty} \simeq \essup_{0<\t<\pi} \int_0^1 (\t+\v)^{2\s-2(\alpha+1)} 
	\frac{\v^{2\alpha+1}}{\v^{2\s}\log\frac{2}{\v}} \, d\v = \int_0^1 \frac{d\v}{\v\log\frac{2}{\v}} = \infty.
$$

Finally, we check that $\mathcal{T}_4$ is not of restricted weak type $(p,q)$ if 
$\frac{1}{q} < \frac{1}{p}-\frac{\sigma}{\delta}$. By the positive results justified above and an
\emph{au contraire} argument involving the interpolation theorem for Lorentz spaces invoked in (C),
it is enough to ensure that $\mathcal{T}_4$ is not of strong type $(p,q)$ if
$\frac{1}{q} < \frac{1}{p}-\frac{\sigma}{\delta}$.
Indeed, if $\mathcal{T}_4$ were of restricted weak type $(p,q)$ for some $p$ and $q$ such that
$\frac{1}q < \frac{1}p -\frac{\s}{\delta}$, then by interpolation with a strong type pair satisfying
$\frac{1}q = \frac{1}p - \frac{\s}{\delta}$, $p>1$, $q < \infty$, $\mathcal{T}_4$ would be of strong type
$(\widetilde{p},\widetilde{q})$ for some $\widetilde{p}$ and $\widetilde{q}$ satisfying
$\frac{1}{\widetilde{q}} < \frac{1}{\widetilde{p}}-\frac{\s}{\delta}$.

Recall that $\s < \delta = \alpha+1$ and assume that $\frac{1}q < \frac{1}p - \frac{\s}{\delta}$. 
Take $f(\v) = \v^A \chi_{(0,1)}(\v)$ with 
$A = -\frac{2\delta}p + 2\delta\varepsilon$, with $\varepsilon >0$ such that 
$\varepsilon < \frac{1}p-\frac{\s}{\delta}-\frac{1}q$. Then 
$$
\|f\|_p^p \simeq \int_0^1 \v^{Ap+2\alpha+1}\, d\v = \int_0^1 \v^{-1+2\delta p \varepsilon} \, d\v < \infty
$$
and hence $f\in L^p(d\mu_{\ab})$. We will show that $\mathcal{T}_4 f \notin L^q(d\mu_{\ab})$.
Let $\t \in (0,1)$. Observe that 
$$
\mathcal{T}_4 f(\t) \simeq \int_0^1 (\t+\v)^{2\s-2(\alpha+1)} \v^{A+2\alpha+1}\, d\v
	= \t^{2\s+A} \int_0^{1/\t} \frac{\v^{2\alpha+1+A}\, d\v}{(1+\v)^{2(\alpha+1)-2\s}}. 
$$
Since $1/\t > 1$, the last integral is certainly larger than a constant. Thus we get
$$
\mathcal{T}_4 f(\t) \gtrsim \t^{2\s+A}, \qquad \t \in (0,1). 
$$
If $q=\infty$, it suffices to observe that $2\s + A < 0$ (see above). For $q<\infty$, we write
$$
\|\mathcal{T}_4 f\|_q^q \gtrsim \int_0^1 \t^{(2\s+A)q+2\alpha+1}\, d\t.
$$
The last integral is infinite since $\frac{1}q < \frac{1}p-\frac{\s}{\delta}-\varepsilon 
= -\frac{A}{2\delta} - \frac{\s}{\delta}$ and consequently $(2\s+A)q + 2\alpha+1 < -1$.
The conclusion follows.

\noindent \textbf{Analysis of $\boldsymbol{\mathcal{T}_6}$.}
Let $\s < 1/2$. We begin with showing that $\mathcal{T}_6$ has all the asserted positive mapping properties.
A crucial observation in this direction is that $\mathcal{T}_6$ 
is comparable with $U_{2\s}$, see \eqref{Uxi}, in the sense that
$$
\mathcal{T}_6 f(\t) \simeq U_{2\s}f(\t), \qquad f \ge 0, \quad \t \in (0,\pi),
$$
and hence these operators have exactly the same $L^p-L^q$ mapping properties.
Indeed, by \cite[Lemma 4.2]{NS1} one has
$$
\mu_{\ab}\big( B(\t,|\t-\v|)\big) \simeq |\t-\v| (\t+\v)^{2\alpha+1}(2\pi-\t-\v)^{2\beta+1},
	\qquad \t,\v \in (0,\pi),
$$
so the kernels of $\mathcal{T}_6$ and $U_{2\s}$ are comparable.
Moreover, in view of the above estimate, 
$$
\mu_{\ab}(B(\t,r))\simeq r (r+\t)^{2\alpha+1} (r+\pi-\t)^{2\beta+1}, \qquad \t \in (0,\pi), \quad 0<r<\pi,
$$
and we see that $\mu_{\ab}(B_r) \gtrsim r^{2\delta}$ for any ball in $(0,\pi)$ of radius $r<\pi$.
Applying now Lemma \ref{lem:Ixi} we conclude that $U_{2\s}$, and hence also $\mathcal{T}_6$, is of weak type
$(1,\frac{\delta}{\delta-\sigma})$.

Next we claim that $\mathcal{T}_6$ is of strong type $(1,q)$ for $1\le q < \frac{\delta}{\delta-\s}$. 
If this is true then, in view of (B), (A) and (C), we get all the remaining
positive results for $\mathcal{T}_6$. 
To prove the claim, by Minkowski's integral inequality it is enough to ensure that
\begin{equation*}
\sup_{0<\v<\pi} \int_0^{\pi} \big(\mathcal{K}_6(\t,\v)\big)^q\, d\mu_{\ab}(\t) < \infty.
\end{equation*}
For symmetry reasons we may restrict the last integration to $\t < \pi/2$. Then we can write
\begin{align*}
\int_0^{\pi/2} \big(\mathcal{K}_6(\t,\v)\big)^q\, d\mu_{\ab}(\t) & \simeq
\int_0^{\pi/2} (\t+\v)^{-(2\alpha+1)q}|\t-\v|^{(2\s-1)q} \t^{2\alpha+1}\, d\t \\
& \le \bigg\{ \int_0^{\v/2} + \int_{\v/2}^{2\v} + \int_{2\v}^{2\pi}\bigg\} \,
	(\t+\v)^{-(2\alpha+1)q}|\t-\v|^{(2\s-1)q} \t^{2\alpha+1}\, d\t \\
& \equiv \mathcal{J}_1 + \mathcal{J}_2 + \mathcal{J}_3.
\end{align*}
We now estimate each of the three integrals uniformly in $\v\in (0,\pi)$. We have
$$
\mathcal{J}_1 \simeq \v^{-(2\alpha+1)q} \v^{(2\s-1)q} \int_0^{\v/2} \t^{2\alpha+1}\, d\t \simeq
	\v^{(2\s-2\alpha-2)q+2\alpha+2} \lesssim 1.
$$
The last bound holds because the condition $q < \frac{\delta}{\delta-\s}$ implies
$(2\s-2\alpha-2)q+2\alpha+2>0$. Further,
$$
\mathcal{J}_2 \simeq \v^{-(2\alpha+1)q} \v^{2\alpha+1} \int_{\v/2}^{2\v} |\t-\v|^{(2\s-1)q}\, d\t \simeq
	\v^{(2\s-2\alpha-2)q+2\alpha+2} \lesssim 1,
$$
where we used the fact that $(2\s-1)q > (2\s-1)\frac{\delta}{\delta-\s}\ge (2\s-1)\frac{1/2}{1/2-\s}=-1$.
Finally, 
$$
\mathcal{J}_3 \simeq \int_{2\v}^{2\pi} \t^{(2\s-2\alpha-2)q+2\alpha+1}\, d\t \lesssim 1.
$$
The claim follows.

Passing to negative results, we observe that 
after neglecting the characteristic functions $\chi_{\{\s>1/2\}}$ and $\chi_{\{\s<1/2\}}$, the
kernel $\mathcal{K}_6(\t,\v)$ is controlled from below by the kernel $\mathcal{K}_4(\t,\v)$.
Thus all the counterexamples given in the analysis of $\mathcal{T}_4$ for the case
$\s< \delta = \alpha+1$ are valid also in the present situation, assuming that $\delta > 1/2$
(note that the condition $\s>1/2$ was irrelevant for the counterexamples related to $\mathcal{T}_4$).

We now give counterexamples for the case $\delta =1/2$ (notice that this means that
$\ab \le -1/2$). This situation is different from that for $\delta > 1/2$ since now the bad behavior is
caused by the factor $|\t-\v|^{2\s-1}$ rather than the endpoint behavior of the kernel.
We follow the strategy from the analysis of $\mathcal{T}_4$ that reduces the task to giving 
two particular counterexamples.

Let us first disprove the weak (strong) type $(\frac{1}{2\s},\infty)$. Take
$f(\v) = \chi_{(1/2,1)}(\v)/((1-\v)^{2\s}\log\frac{2}{1-\v})$. Then $f \in L^{{1}/({2\s})}(d\mu_{\ab})$,
but 
$$
\|\mathcal{T}_6 f\|_{\infty} \simeq \essup_{0<\t<\pi}
	\int_{1/2}^1 \frac{|\t-\v|^{2\s-1}\, d\v}{(1-\v)^{2\s}\log\frac{1}{1-\v}}
	\ge \int_{1/2}^1 \frac{d\v}{(1-\v)\log\frac{1}{1-\v}} = \infty.
$$

Next, we disprove strong type $(p,q)$ when $\frac{1}q < \frac{1}p - 2\s$.
Consider $f(\v) = (1-\v)^A \chi_{(1/2,1)}(\v)$ with $A=-\frac{1}p+\varepsilon$, where $\varepsilon > 0$
is such that $\varepsilon < \frac{1}p - 2\s -\frac{1}q$. Then 
$$
\|f\|_p^p \simeq \int_{1/2}^1 (1-\v)^{-1+\varepsilon p}\, d\v < \infty,
$$
so $f \in L^p(d\mu_{\ab})$. We will show that $\mathcal{T}_6 f \notin L^q(d\mu_{\ab})$.
Let $\t \in (1,3/2)$. Changing the variable of integration we get
$$
\mathcal{T}_6 f(\t) \simeq \int_{1/2}^1 (\t-\v)^{2\s-1} (1-\v)^A\, d\v \simeq (\t-1)^{2\s+A}
	\int_0^{1/(2(\t-1))} \frac{\v^A\, d\v}{(1+\v)^{1-2\s}}.
$$
Since $1/(2(\t-1))>1$, the last integral is larger than a constant. Hence
$$
\mathcal{T}_6 f(\t) \gtrsim (\t-1)^{2\s+A}, \qquad \t \in (1,3/2).
$$
We see that $\mathcal{T}_6 f$ is not in $L^{\infty}$ since $2\s+A<0$. 
Neither it belongs to $L^q(d\mu_{\ab})$, $q< \infty$, because $(2\s+A)q< -1$ and consequently
$$
\|\mathcal{T}_6 f\|_q^{q} \gtrsim \int_1^{3/2} (\t-1)^{(2\s+A)q}\, d\t = \infty.
$$

\noindent \textbf{Analysis of $\boldsymbol{\mathcal{T}_5}$.}
We first consider the case $\delta=1/2=\s$. Observe that $\mathcal{K}_5(\t,\v)$ 
is controlled from above by the kernel $\mathcal{K}_6(\t,\v)$ with any fixed $\s < 1/2$. 
Therefore we can deduce from the already proved results for $\mathcal{T}_6$ that $\mathcal{T}_5$ 
is of strong type $(p,q)\neq (1,\infty)$. On the other hand, $\mathcal{T}_5$ is not of
restricted weak type $(1,\infty)$. To see this, let $f_{\varepsilon} = \chi_{(1-\varepsilon,1)}$ with
$\varepsilon >0$ small. Then $\|f_{\varepsilon}\|_{1} \simeq \varepsilon$ and
$$
\|\mathcal{T}_5 f_{\varepsilon}\|_{\infty} \simeq \essup_{0<\t<\pi} \int_{1-\varepsilon}^1 
	\log \frac{\pi}{|\t-\v|} \, d\v
	\ge \int_{1-\varepsilon}^{1} \log\frac{\pi}{1-\v}\, d\v \simeq \varepsilon \log\frac{\pi}{\varepsilon},
$$
and the conclusion follows by letting $\varepsilon \to 0$.

Assume next that $\delta> \s = 1/2$. For symmetry reasons, we may and do restrict to the case
$\alpha \ge \beta$; in particular, $\delta = \alpha+1$. We will show that $\mathcal{T}_5$ 
has the mapping properties from item (iii) of the theorem. 
Taking into account the above mentioned majorization
by the kernel $\mathcal{K}_6(\t,\v)$ and the positive results for $\mathcal{T}_6$, 
we see that to obtain the positive results for $\mathcal{T}_5$ it remains to analyze pairs $(p,q)$ 
satisfying $\frac{1}q = \frac{1}p -\frac{\s}{\delta}$.
By (C), this task can be reduced to showing that $\mathcal{T}_5$ 
is of weak type $(1,\frac{\delta}{\delta-\s})$.

To proceed, observe that the logarithmic factor in $\mathcal{K}_5(\t,\v)$ can be large only 
if $\t$ and $\v$ are comparable and simultaneously $\pi -\t$ and $\pi - \v$ are comparable; 
otherwise the logarithm is controlled by a constant. Thus
\begin{align*}
\mathcal{K}_5(\t,\v)  & \lesssim (\t+\v)^{-2\alpha-1} (2\pi - \t -\v)^{-2\beta-1} 
		+ \chi_{\{2\t/3<\v<3\t/2,\, 2(\pi-\t)/{3}<\pi-\v<3(\pi-\t)/2\}} 
		\mathcal{K}_5(\t,\v)\\
& \equiv \mathcal{K}_{5,1}(\t,\v) + \mathcal{K}_{5,2}(\t,\v).
\end{align*}
The operator $\mathcal{T}_{5,1}$ given by the kernel $\mathcal{K}_{5,1}(\t,\v)$ 
is of weak type $(1,\frac{\delta}{\delta-\s})$,
see the analysis of $\mathcal{T}_4$ (the argument given there is valid also for $\s=1/2$). 
As for the operator $\mathcal{T}_{5,2}$ defined by $\mathcal{K}_{5,2}(\t,\v)$, 
we will prove that it is even strong type $(1,\frac{\delta}{\delta-\s})$.
To achieve this, in view of (D) (notice that $\mathcal{K}_{5,2}(\t,\v)$ is symmetric and nonnegative), 
it is enough to verify that $\mathcal{T}_{5,2}$ is of strong type $(\frac{\delta}{\s},\infty)$.

We have
\begin{align*}
\mathcal{K}_{5,2}(\t,\v) & \lesssim \chi_{(0,\pi/2]}(\t) \chi_{\{2\t/3<\v<3\t/2\}} (\t+\v)^{-2\alpha-1}
	\log\frac{2\pi(\t+\v)}{|\t-\v|} \\
& \quad + \chi_{(\pi/2,\pi)}(\t) \chi_{\{2(\pi-\t)/{3}<\pi-\v<3(\pi-\t)/2\}}
	(2\pi-\t-\v)^{-2\beta-1} \log\frac{2\pi(2\pi-\t-\v)}{|\pi-\t - (\pi-\v)|}\\
& \equiv \mathcal{K}_{5,2,0}(\t,\v) + \mathcal{K}_{5,2,\pi}(\t,\v).
\end{align*}
Let $\mathcal{T}_{5,2,0}$ and $\mathcal{T}_{5,2,\pi}$ be the operators given by 
$\mathcal{K}_{5,2,0}(\t,\v)$ and $\mathcal{K}_{5,2,\pi}(\t,\v)$, respectively. 
Then, by H\"older's inequality,
\begin{align*}
\mathcal{T}_{5,2,0} f(\t) & \simeq \chi_{(0,\pi/2]}(\t) \, \t^{-2\alpha-1}
	\int_{2\t/3}^{3\t/2} \log\frac{6\pi\t}{|\t-\v|} f(\v) \v^{2\alpha+1} \, d\v \\
& \lesssim \chi_{(0,\pi/2]}(\t) \, \t^{-2\alpha-1} \|f\|_{\delta/\s} \Bigg( \int_{2\t/3}^{3\t/2}
	\bigg( \log\frac{6\pi}{|1-\v/\t|}\bigg)^{\delta/(\delta-\s)} \v^{2\alpha+1}\, d\v \Bigg)^{1-\s/\delta}.
\end{align*}
To estimate the last integral we change the variable of integration and get
$$
\int_{2\t/3}^{3\t/2}
	\bigg( \log\frac{6\pi}{|1-\v/\t|}\bigg)^{\delta/(\delta-\s)} \v^{2\alpha+1}\, d\v 
	= \t^{2\alpha+2}\int_{2/3}^{3/2} \bigg( \log\frac{6\pi}{|1-\psi|}\bigg)^{\delta/(\delta-\s)}
		\psi^{2\alpha+1}\, d\psi \simeq \t^{2\alpha+2}.
$$
Consequently, 
$$
\mathcal{T}_{5,2,0} f(\t) \lesssim \chi_{(0,\pi/2]}(\t) \, \t^{-2\alpha-1} 
\|f\|_{\delta/\s} \,\t^{(2\alpha+2)(1-\s/\delta)} \le \|f\|_{\delta/\s}, \qquad \t \in (0,\pi).
$$
It follows that $\mathcal{T}_{5,2,0}$ is of strong type $(\frac{\delta}{\s},\infty) = (2\alpha+2,\infty)$.
The same arguments apply to $\mathcal{T}_{5,2,\pi}$ and give strong type 
$(2\beta+2,\infty)$. Since $\beta \le \alpha$,
this implies strong type $(2\alpha+2,\infty)$ for $\mathcal{T}_{5,2,\pi}$, see (A). 
We conclude that $\mathcal{T}_{5,2}$ is of strong type $(\frac{\delta}{\s},\infty)$, as desired.

Passing to negative results for the case $\delta>\s=1/2$, 
we observe that $\mathcal{K}_5(\t,\v)$ is controlled from below by the kernel $\mathcal{K}_4(\t,\v)$
with any $\s>1/2$ fixed. Therefore the counterexamples given in the analysis of $\mathcal{T}_4$ 
imply that $\mathcal{T}_5$ is not of
restricted weak type $(p,q)$ if $\frac{1}q < \frac{1}p -\frac{\s}{\delta}$. It remains to disprove
the weak (strong) type $(\frac{\delta}{\s},\infty)= (2\alpha+2,\infty)$. 
Then automatically the strong type
$(1,\frac{\delta}{\delta-\s})$ will also be disproved, see (B). 
Take $f(\v) = \chi_{(0,1)}(\v)/(\v\log\frac{2}{\v})$. Then $f \in L^{2\alpha+2}(d\mu_{\ab})$. But
$$
\|\mathcal{T}_5 f\|_{\infty} \simeq \essup_{0<\t<\pi} \int_0^1 (\t+\v)^{-2\alpha-1} 
	\log\frac{(\t+\v)(2\pi-\t-\v)}{|\t-\v|} \frac{\v^{2\alpha+1}}{\v\log\frac{2}{\v}}\, d\v
		\ge \int_0^1 \frac{d\v}{\v\log\frac{2}{\v}} = \infty.
$$
This finishes the analysis of $\mathcal{T}_5$.

The proof of Theorem \ref{thm:LpLqJac} is complete.
\end{proof}

\subsection{$\boldsymbol{L^p-L^q}$ estimates in the Jacobi trigonometric `function' setting}
Similarly as for the proof of Theorem \ref{thm:LpLqJac},
to prove Theorem \ref{thm:LpLqJacL} we decompose, in the sense of $\simeq$, the kernel
$\mathbb{K}_{\s}^{\ab}(\t,\v)$ according to \eqref{Kleb} and the estimate of Theorem \ref{thm:potker}.
We get
$$
\mathbb{K}_{\s}^{\ab}(\t,\v) \simeq \sum_{i=1}^{6} \mathbb{K}_i(\t,\v),
$$
where
\begin{align*}
\mathbb{K}_1(\t,\v) & :=  (\t\v)^{\alpha+1/2}\big((\pi-\t)(\pi-\v)\big)^{\beta+1/2},\\
\mathbb{K}_2(\t,\v) & :=  \chi_{\{\s=\alpha+1\}}(\t\v)^{\alpha+1/2}\big((\pi-\t)(\pi-\v)\big)^{\beta+1/2}
	\log\frac{2\pi}{\t+\v},\\
\mathbb{K}_3(\t,\v) & :=  \chi_{\{\s=\beta+1\}}(\t\v)^{\alpha+1/2}\big((\pi-\t)(\pi-\v)\big)^{\beta+1/2}
	\log\frac{2\pi}{2\pi-\t-\v},\\
\mathbb{K}_4(\t,\v) & := \chi_{\{\s>1/2\}}\bigg[ \frac{\t\v}{(\t+\v)^2}\bigg]^{\alpha+1/2} 
	\bigg[ \frac{(\pi-\t)(\pi-\v)}{(2\pi-\t-\v)^2}\bigg]^{\beta+1/2} [(\t+\v)(2\pi-\t-\v)]^{2\s-1},\\
\mathbb{K}_5(\t,\v) & := \chi_{\{\s=1/2\}}\bigg[ \frac{\t\v}{(\t+\v)^2}\bigg]^{\alpha+1/2} 
	\bigg[ \frac{(\pi-\t)(\pi-\v)}{(2\pi-\t-\v)^2}\bigg]^{\beta+1/2} \log\frac{(\t+\v)(2\pi-\t-\v)}{|\t-\v|},\\
\mathbb{K}_6(\t,\v) & := \chi_{\{\s<1/2\}}\bigg[ \frac{\t\v}{(\t+\v)^2}\bigg]^{\alpha+1/2} 
	\bigg[ \frac{(\pi-\t)(\pi-\v)}{(2\pi-\t-\v)^2}\bigg]^{\beta+1/2} |\t-\v|^{2\s-1}.
\end{align*}
We denote the corresponding integral operators by $\mathbb{T}_i$, $i=1,\ldots,6$,
$$
\mathbb{T}_i f(\t) = \int_0^{\pi} \mathbb{K}_i(\t,\v) f(\v)\, d\v.
$$
Since all the kernels are nonnegative, in what follows we may and always do assume that $f \ge 0$.
Further, to show any of the asserted positive mapping properties of $\mathbb{I}_{\s}^{\ab}$, it is enough
to do the same for each $\mathbb{T}_i$, $i=1,\ldots,6$. On the other hand, to disprove one of the mapping
properties for $\mathbb{I}_{\s}^{\ab}$, it suffices to check that it fails in case of a particular 
$\mathbb{T}_i$.

While reading the proof below, we advise the reader to take advantage of Figures \ref{fig2}-\ref{fig4}
and also to draw own pictures for cases not covered by Figures \ref{fig2}-\ref{fig4}.

\begin{proof}[Proof of Theorem {\ref{thm:LpLqJacL}}. Part (a)]
Throughout we always assume that $1\le p,q \le \infty$ and that $\kappa \ge 0$, i.e. $\ab \ge -1/2$. 
We will verify the following mapping properties of $\mathbb{T}_i$, $i=1,\ldots,6$, which altogether
imply items (a1)-(a3) and their sharpness.
\begin{itemize}
\item $\mathbb{T}_1$ is of strong type $(p,q)$ for all $p$ and $q$.
\item $\mathbb{T}_2$, being nontrivial only for $\s=\alpha+1$, is of strong type $(p,q)$ for all $p$
	and $q$, except when $\s=1/2$ and $\alpha=-1/2$; in the latter case $\mathbb{T}_2$ is of strong type
	$(p,q)\neq (1,\infty)$.
\item $\mathbb{T}_3$, being nontrivial only for $\s=\beta+1$, is of strong type $(p,q)$ for all $p$
	and $q$, except when $\s=1/2$ and $\beta=-1/2$; in the latter case $\mathbb{T}_3$ is of strong type
	$(p,q)\neq (1,\infty)$.
\item $\mathbb{T}_4$, being nontrivial only for $\s>1/2$, is of strong type $(p,q)$ for all $p$ and $q$.
\item $\mathbb{T}_5$, being nontrivial only for $\s=1/2$, is in this case of strong type 
	$(p,q)\neq (1,\infty)$, and not of restricted weak type $(1,\infty)$.
\item $\mathbb{T}_6$, being nontrivial only for $\s<1/2$, satisfies in this case the following, 
	see Figure \ref{fig2}. $\mathbb{T}_6$ is of strong type $(p,q)$ if 
	$\frac{1}q \ge \frac{1}p-2\s$ and $(p,q)\neq (1,\frac{1}{1-2\s}),(\frac{1}{2\s},\infty)$,
	of weak type $(1,\frac{1}{1-2\s})$, and of restricted weak type $(\frac{1}{2\s},\infty)$.
	As for negative results, $\mathbb{T}_6$ is not of strong type $(1,\frac{1}{1-2\s})$, not of weak type
	$(\frac{1}{2\s},\infty)$, and not of restricted weak type $(p,q)$ when $\frac{1}q<\frac{1}p-2\s$.
\end{itemize}

\noindent \textbf{Analysis of $\boldsymbol{\mathbb{T}_1}$.}
Since $\mathbb{K}_1(\t,\v) \lesssim 1$, the conclusion follows.

\noindent \textbf{Analysis of $\boldsymbol{\mathbb{T}_2}$.}
We have
$$
\mathbb{K}_2(\t,\v) \lesssim (\t\v)^{\alpha+1/2}\log\frac{2\pi}{\t+\v}.
$$
If $\alpha>-1/2$, then the right-hand side here is controlled by a constant and hence $\mathbb{T}_2$
is of strong type $(p,q)$ for all $p$ and $q$. If $\alpha=-1/2$, then
$\mathbb{K}_2(\t,\v) \lesssim \log\frac{2\pi}{\t+\v}$ and, as we saw in the proof of 
Theorem \ref{thm:LpLqJac} (see the analysis of $\mathcal{T}_2$ with $\alpha=\beta=-1/2$), 
$\mathbb{T}_2$ is of strong type $(p,q)$ except for $(p,q)=(1,\infty)$.

\noindent \textbf{Analysis of $\boldsymbol{\mathbb{T}_3}$.}
We either use the same arguments as in case of $\mathbb{T}_2$, or conclude the mapping properties for
$\mathbb{T}_3$ from those for $\mathbb{T}_2$ by replacing $\t$ by $\pi-\t$, $\v$ by $\pi-\v$, and
exchanging the roles of $\alpha$ and $\beta$.

\noindent \textbf{Analysis of $\boldsymbol{\mathbb{T}_4}$, $\boldsymbol{\mathbb{T}_5}$ and
	$\boldsymbol{\mathbb{T}_6}$.}
	Observe that for $i=4,5,6$ we have
	$$
		\mathbb{K}_i(\t,\v) \le \mathcal{K}_i(\t,\v),
	$$
where the kernels $\mathcal{K}_i(\t,\v)$ on the right-hand side here are defined in Section \ref{sssec}
and taken with $\alpha=\beta=-1/2$. Moreover, $d\mu_{-1/2,-1/2}(\t) = d\t$. Therefore $\mathbb{T}_i$
are for $i=4,5,6$ controlled, respectively, by $\mathcal{T}_i$ from the proof of Theorem \ref{thm:LpLqJac}
with $\alpha$ and $\beta$ specified to be $-1/2$. Consequently, the positive mapping properties
of the latter operators stated and verified in the proof of Theorem \ref{thm:LpLqJac} are inherited,
respectively, by $\mathbb{T}_i$, $i=4,5,6$. This gives the asserted positive results for 
$\mathbb{T}_4$, $\mathbb{T}_5$ and $\mathbb{T}_6$.

On the other hand, for $\t$ and $\v$ separated from the endpoints of $(0,\pi)$, the kernels
$\mathbb{K}_5(\t,\v)$ and $\mathbb{K}_6(\t,\v)$ are comparable, respectively, with
$\mathcal{K}_5(\t,\v)$ and $\mathcal{K}_6(\t,\v)$ taken with $\alpha=\beta=-1/2$. Thus the relevant
counterexamples and arguments from the proof of Theorem \ref{thm:LpLqJac}, the analysis of
$\mathcal{T}_5$ and $\mathcal{T}_6$ with $\alpha=\beta=-1/2$ and hence $\delta=1/2$, work in
the present situation and deliver the desired negative results for $\mathbb{T}_5$ and $\mathbb{T}_6$.

The proof of part (a) in Theorem \ref{thm:LpLqJacL} is complete.
\end{proof}

\begin{proof}[Proof of Theorem {\ref{thm:LpLqJacL}}. Part (b)]
Recall that we consider $1\le p,q \le \infty$. Throughout this part of the proof we 
always assume that $\kappa < 0$, that is $\alpha \wedge \beta < -1/2$.
We will analyze separately the relevant mapping properties of $\mathbb{T}_i$, $i=1,\ldots,6$.
More precisely, we will prove the following items, which altogether imply (b1)-(b3) and their sharpness.
\begin{itemize}
\item $\mathbb{T}_1$ is (see Figure \ref{fig3}) of strong type $(p,q)$ if $\frac{1}{p}<1+\kappa$ and
	$\frac{1}q > -\kappa$, of weak type $(p,\frac{1}{-\kappa})$ for $\frac{1}p < 1+\kappa$, and of restricted
	weak type $(\frac{1}{1+\kappa},q)$ when $\frac{1}q \ge -\kappa$. On the other hand, $\mathbb{T}_1$ is
	not of strong type $(p,\frac{1}{-\kappa})$ for $\frac{1}{p} < 1+\kappa$, not of weak type
	$(\frac{1}{1+\kappa},q)$ for $\frac{1}q \ge -\kappa$, and not of restricted weak type $(p,q)$ if
	$\frac{1}{p} > 1+\kappa$ or $\frac{1}q < -\kappa$.
\item $\mathbb{T}_2$, being nontrivial only for $\s=\alpha+1$, has then all the positive mapping properties
	indicated above for $\mathbb{T}_1$, except for that $\mathbb{T}_2$ fails to be of restricted weak type
	$(\frac{1}{1+\kappa},\frac{1}{-\kappa})$ in case $\alpha \le \beta$.
\item $\mathbb{T}_3$, being nontrivial only for $\s=\beta+1$, has then all the positive mapping properties
	indicated above for $\mathbb{T}_1$, except for that $\mathbb{T}_3$ fails to be of restricted weak type
	$(\frac{1}{1+\kappa},\frac{1}{-\kappa})$ in case $\beta \le \alpha$.
\item $\mathbb{T}_4$, being nontrivial only for $\s>1/2$, has the positive mapping properties indicated
	above for $\mathbb{T}_1$. 
\item $\mathbb{T}_5$, being nontrivial only for $\s=1/2$, also has the positive mapping properties
	indicated above for $\mathbb{T}_1$. 
\item $\mathbb{T}_6$, being nontrivial only for $\s<1/2$, satisfies in this case the following:
	\begin{itemize}
		\item[$\star$] 
			if $\s>\kappa+1/2$, then $\mathbb{T}_6$ has the positive mapping properties of $\mathbb{T}_1$
			indicated above;
		\item[$\star$] 
			if $\s=\kappa+1/2$, then $\mathbb{T}_6$ has the positive mapping properties of $\mathbb{T}_1$
			indicated above, excluding the pair $(p,q)=(\frac{1}{1+\kappa},\frac{1}{-\kappa})$;
		\item[$\star$] 
			if $\s<\kappa+1/2$, then $\mathbb{T}_6$ is (see Figure \ref{fig4}) of strong type $(p,q)$
			when $\frac{1}p< 1+\kappa$, $\frac{1}q> -\kappa$ and $\frac{1}q \ge \frac{1}p-2\s$, of
			weak type $(p,\frac{1}{-\kappa})$ for $\frac{1}p<2\s-\kappa$, and of restricted weak type
			$(\frac{1}{2\s-\kappa},\frac{1}{-\kappa})$ and $(\frac{1}{1+\kappa},q)$ 
			for $\frac{1}q \ge 1+\kappa -2\s$; concerning negative results, $\mathbb{T}_6$ is 
			not of weak type $(\frac{1}{2\s-\kappa},\frac{1}{-\kappa})$
			and not of restricted weak type $(p,q)$ if $\frac{1}q<\frac{1}p-2\s$.
	\end{itemize}
\end{itemize}

\noindent \textbf{Analysis of $\boldsymbol{\mathbb{T}_1}$.}
Define the kernel
$$
\widetilde{\mathbb{K}}_1(\t,\v) := (\t\v)^{\kappa} \big[(\pi-\t)(\pi-\v)\big]^{\kappa},
$$
which is symmetric with respect to $\t=\pi /2$ and $\v=\pi/2$. Since
$$
\mathbb{K}_1(\t,\v) \lesssim \widetilde{\mathbb{K}}_1(\t,\v), 
$$
it is enough to prove the above mentioned positive mapping properties for the operator
$\widetilde{\mathbb{T}}_1$ associated to $\widetilde{\mathbb{K}}_1(\t,\v)$, rather than $\mathbb{T}_1$.

Let $p$ and $q$ be such that $\frac{1}p < 1+ \kappa$ and $\frac{1}q > -\kappa$.
Using H\"older's inequality, we get
\begin{equation} 
\widetilde{\mathbb{T}}_1 f(\t)  
	\le [\t(\pi-\t)]^{\kappa} \|f\|_{p} \bigg( \int_0^{\pi} 
	[\v(\pi-\v)]^{\kappa p'} \, d\v\bigg)^{1/p'} 
 \lesssim [\t(\pi-\t)]^{\kappa} \|f\|_{p} \label{e1};
\end{equation} 
here the $L^{p'}$ norm is indeed finite because the condition $\frac{1}p < 1+\kappa$ implies
$\kappa p' > -1$. Therefore
$$
\|\widetilde{\mathbb{T}}_1 f\|_{q} \lesssim \|f\|_{p} \bigg( \int_0^{\pi} [\t(\pi-\t)]^{\kappa q}\, 
	d\t\bigg)^{1/q}.
$$
Since  
$\kappa q > -1$, the last integral is finite and it follows that $\widetilde{\mathbb{T}}_1$ 
is of strong type $(p,q)$.

We now verify the weak type $(p,\frac{1}{-\kappa})$ for $\frac{1}p < 1+\kappa$. We have, see \eqref{e1},
\begin{align*}
\widetilde{\mathbb{T}}_1 f(\t) &
 \lesssim \chi_{(0,\pi/2]}(\t) \, \t^{\kappa} \|f\|_{p} 
 + \chi_{(\pi/2,\pi)}(\t)\,(\pi-\t)^{\kappa}\|f\|_{p} \\
 & \equiv \widetilde{\mathbb{T}}_{1,0} f(\t) + \widetilde{\mathbb{T}}_{1,\pi} f(\t).
\end{align*}
Then, for $\lambda>0$,
$$
\big|\big\{\widetilde{\mathbb{T}}_{1,0}f > \lambda\big\}\big| 
	= \Bigg|\Bigg\{ \t \in \Big(0,\frac{\pi}2\Big] : \t < \bigg(\frac{\lambda}{\|f\|_p}\bigg)^{1/\kappa}
		\Bigg\}\Bigg|
	\le \bigg(\frac{\|f\|_p}{\lambda}\bigg)^{-1/\kappa}.
$$
Since treatment of $\widetilde{\mathbb{T}}_{1,\pi}$ is analogous, the conclusion follows.

Next we prove the restricted weak type $(\frac{1}{1+\kappa},q)$ for $\frac{1}q \ge -\kappa$.
Let $E$ be a measurable subset of $(0,\pi)$. We have
\begin{equation} \label{e2}
\widetilde{\mathbb{T}}_1 \chi_{E}(\t) 
	\lesssim [\t(\pi-\t)]^{\kappa} \int_0^{|E|} \v^{\kappa}\, d\v
	\simeq [\t(\pi-\t)]^{\kappa} \|\chi_E\|_{1/(1+\kappa)}.
\end{equation}
Since $\frac{1}q \ge -\kappa$, this gives
$$
\widetilde{\mathbb{T}}_1 \chi_{E}(\t) \lesssim [\t(\pi-\t)]^{-1/q} \|\chi_{E}\|_{1/(1+\kappa)}. 
$$
Proceeding as in case of the weak type above, we see that
$$
\big|\big\{\widetilde{\mathbb{T}}_1 \chi_{E}> \lambda\big\}\big| \lesssim
 \bigg(\frac{\|\chi_{E}\|_{1/(1+\kappa)}}{\lambda}\bigg)^{q},
$$
uniformly in $E$ and $\lambda > 0$.

Passing to negative results, assume without any loss of generality that $\alpha \le \beta$, 
so that $\kappa = \alpha+1/2$.
We first observe that $\mathbb{T}_1$ is not of strong type $(p,\frac{1}{-\kappa})$
if $\frac{1}p < 1+\kappa$. Indeed, taking $f=\chi_{(0,1)}\in L^p$, we have
$$
\mathbb{T}_1 f(\t) \simeq \t^{\kappa} \int_0^1 \v^{\kappa}\, d\v \simeq \t^{\kappa}, \qquad \t \in (0,1),
$$
and hence
$$
\|\mathbb{T}_1 f\|_{-1/\kappa}^{-1/\kappa} \gtrsim \int_0^1 \t^{-1}\, d\t = \infty.
$$

Next, we disprove the weak type $(\frac{1}{1+\kappa},q)$ for $\frac{1}q \ge -\kappa$.
Let $f(\v) = \chi_{(0,1)}(\v)/(\v^{1+\kappa}\log\frac{2}\v)$. Then $f\in L^{1/(1+\kappa)}$, but
$$
\mathbb{T}_1 f(\t) \simeq \t^{\alpha+1/2}(\pi-\t)^{\beta+1/2} 
	\int_0^1 \frac{d\v}{\v \log\frac{2}{\v}} = \infty. 
$$

Finally, we show that the conditions $\frac{1}p \le 1+\kappa$ and $\frac{1}q \ge -\kappa$ are
necessary for $\mathbb{T}_1$ to be of restricted weak type $(p,q)$. 
Take $f_{\varepsilon}=\chi_{(0,\varepsilon)}$ with $\varepsilon < 1$.
Then $\|f_{\varepsilon}\|_p = \varepsilon^{1/p}$ and
$$
\mathbb{T}_1 f_{\varepsilon}(\t) \simeq \t^{\kappa} \int_0^{\varepsilon} \v^{\kappa}\, d\v
	= \t^{\kappa} \varepsilon^{\kappa+1}, \qquad \t \in (0,1).
$$
Therefore, for $\lambda > 0$,
$$
|\{\mathbb{T}_1 f_{\varepsilon} > \lambda\}| 
	\ge |\{\t \le 1 : c\,\t^{\kappa}\varepsilon^{\kappa+1}> \lambda\}|
	= \big|\big\{\t \le 1 : \t < (c\,\varepsilon^{\kappa+1}/\lambda)^{-1/\kappa}\big\}\big|
$$
with $c>0$ independent of $\varepsilon$ and $\lambda$. This gives
$$
|\{\mathbb{T}_1 f_{\varepsilon}> \lambda\}| 
	\ge \Big( \frac{c\,\varepsilon^{\kappa+1}}{\lambda}\Big)^{-1/\kappa},
	\qquad \lambda \ge c\, \varepsilon^{\kappa+1} > 0.
$$
Now we see that the restricted weak type $(p,q)$, $q<\infty$, of $\mathbb{T}_1$ implies
$$
\lambda^{1+1/(\kappa q)} \varepsilon^{-(\kappa+1)/(\kappa q)-1/p} \lesssim 1, \qquad
	\lambda \ge c \,\varepsilon^{\kappa+1} > 0.
$$
This forces $1+1/(\kappa q) \le 0$, i.e.~$\frac{1}q \ge -\kappa$. 
Letting $\lambda = c\,\varepsilon^{\kappa+1}$
we recover also the condition $\frac{1}p \le 1+\kappa$. When $q=\infty$, the weak type estimate for
$\mathbb{T}_1 f_{\varepsilon}$ reads as 
$\|\mathbb{T}_1 f_{\varepsilon}\|_{\infty} \lesssim \varepsilon^{1/p}$, $\varepsilon < 1$,
which means that
$$
\t^{\kappa} \varepsilon^{\kappa+1} \lesssim \varepsilon^{1/p}, \qquad \varepsilon,\t < 1.
$$
Consequently, we must have $0=\frac{1}q \ge -\kappa$ and $\frac{1}p \le 1+\kappa$.

\noindent \textbf{Analysis of $\boldsymbol{\mathbb{T}_2}$.}
Assuming that $\frac{1}p < 1+\kappa$ and using the bound $\log\frac{2\pi}{\t+\v} \le \log\frac{2\pi}{\v}$
we get, see \eqref{e1},
$$
\mathbb{T}_2f(\t) \lesssim [\t(\pi-\t)]^{\kappa} \|f\|_{p}. 
$$
As we already saw, this estimate implies that $\mathbb{T}_2$ is of strong type
$(p,q)$ if $\frac{1}q>-\kappa$, and of weak type $(p,\frac{1}{-\kappa})$. 

Let now $q$ satisfy $\frac{1}q>-\kappa$ and let $E$ be a measurable subset of $(0,\pi)$.
Using the bound $\log\frac{2\pi}{\t+\v} \le \log\frac{2\pi}{\t}$ and estimating similarly as in
\eqref{e2}, we get 
$$
\mathbb{T}_2\chi_{E}(\t) \lesssim \t^{\alpha+1/2}(\pi-\t)^{\beta+1/2} \log\frac{2\pi}{\t}
 \|\chi_E\|_{1/(1+\kappa)}. 
$$
Since $\frac{1}q>-\kappa$, this implies
\begin{equation} \label{e3}
\mathbb{T}_2\chi_{E}(\t) \lesssim [\t(\pi-\t)]^{-1/q} \|\chi_{E}\|_{1/(1+\kappa)},
\end{equation}
which leads to the restricted weak type $(\frac{1}{1+\kappa},q)$, see the analysis of 
$\widetilde{\mathbb{T}}_1$.
If $\beta < \alpha$ and $q=\frac{1}{-\kappa}$, then \eqref{e3} still holds and so in this case 
$\mathbb{T}_2$ is also of restricted weak type $(\frac{1}{1+\kappa},\frac{1}{-\kappa})$.

It remains to disprove the restricted weak type 
$(\frac{1}{1+\kappa},\frac{1}{-\kappa})$ when $\alpha \le \beta$ (i.e.~$\kappa = \alpha+1/2$).
Let $f_{\varepsilon} = \chi_{(0,\varepsilon)}$ for $\varepsilon > 0$ small. 
Then $\|f_{\varepsilon}\|_{1/(1+\kappa)} = \varepsilon^{1+\kappa}$ and 
$$
\mathbb{T}_2 f_{\varepsilon}(\t) 
	\simeq \t^{\kappa} \int_0^{\varepsilon} \v^{\kappa} \log\frac{2\pi}{\t+\v} \, d\v
	\simeq \t^{\kappa} \varepsilon^{1+\kappa} \log\frac{2\pi}{\t+\varepsilon}, \qquad \t \in (0,1).
$$
It follows that for a fixed constant $c>0$ independent of $\varepsilon$ and $\lambda>0$,
$$
|\{\mathbb{T}_2 f_{\varepsilon}> \lambda\}|  \ge 
	\Big| \Big\{ \t< \varepsilon : c\,\t^{\kappa}\varepsilon^{1+\kappa}
	\log\frac{\pi}{\varepsilon} > \lambda \Big\} \Big| 
 = \Bigg| \Bigg\{ \t< \varepsilon : \t < 
	\bigg(\frac{\lambda}{c\,\varepsilon^{1+\kappa}\log\frac{\pi}{\varepsilon}}\bigg)^{1/\kappa}\Bigg\}\Bigg|.
$$
Choosing $\lambda_{\varepsilon} = c \,\varepsilon^{2\kappa +1}\log\frac{\pi}{\varepsilon}$, we get the 
lower bound
$$
|\{\mathbb{T}_2 f_{\varepsilon}> \lambda_{\varepsilon}\}| \ge \varepsilon.
$$
However,
$$
\bigg( \frac{\|f_{\varepsilon}\|_{1/(1+\kappa)}}{\lambda_{\varepsilon}}\bigg)^{-1/\kappa} \simeq
	\varepsilon \Big( \log\frac{\pi}{\varepsilon}\Big)^{1/\kappa},
$$
and since the last expression tends faster to $0$ than $\varepsilon$ itself, the estimate
$$
|\{\mathbb{T}_2 f_{\varepsilon}> \lambda_{\varepsilon}\}| \lesssim 
	\bigg( \frac{\|f_{\varepsilon}\|_{1/(1+\kappa)}}{\lambda_{\varepsilon}}\bigg)^{-1/\kappa}
$$
cannot be uniform in $\varepsilon$ when $\varepsilon \to 0$.

\noindent \textbf{Analysis of $\boldsymbol{\mathbb{T}_3}$.}
See the corresponding comment in the proof of part (a), which remains in force also in the present situation.

\noindent \textbf{Analysis of $\boldsymbol{\mathbb{T}_4}$.}
Observe that
$$
\mathbb{K}_4(\t,\v) \lesssim \bigg[ \frac{\t\v(\pi-\t)(\pi-\v)}{(\t+\v)^2(\pi-\t + \pi-\v)^2}\bigg]^{\kappa}
	\lesssim \widetilde{\mathbb{K}}_1(\t,\v). 
$$
Consequently, $\mathbb{T}_4$ inherits the positive mapping properties of
$\widetilde{\mathbb{T}}_1$ justified above.

\noindent \textbf{Analysis of $\boldsymbol{\mathbb{T}_5}$.}
It can be easily seen that $\mathbb{K}_5(\t,\v)$ is controlled from above, uniformly in
$\t,\v\in (0,\pi)$, by the kernel $\mathbb{K}_6(\t,\v)$ with any fixed $\s<1/2$.
Therefore $\mathbb{T}_5$ inherits the positive mapping properties of $\mathbb{T}_6$ to be proved in
a moment. Choosing $\s$ such that $\kappa + 1/2 < \s < 1/2$, we infer that
$\mathbb{T}_5$ has the positive mapping properties of ${\mathbb{T}}_1$, provided
that what is claimed about $\mathbb{T}_6$ in the beginning of this proof is true.

\noindent \textbf{Analysis of $\boldsymbol{\mathbb{T}_6}$.}
Assume that $\s<1/2$. In order to show the positive results for $\mathbb{T}_6$, we observe that
$$
\mathbb{K}_6(\t,\v) \lesssim \bigg[ \frac{\t\v(\pi-\t)(\pi-\v)}{(\t+\v)^2(\pi-\t + \pi-\v)^2}\bigg]^{\kappa}
	|\t-\v|^{2\s-1}
$$
and consider the dominating kernel on the right-hand side here. Then, for symmetry reasons, we may restrict
to $\t \in (0,\pi/2]$. Thus it is enough to study the kernel
\begin{align*}
\widetilde{\mathbb{K}}_6(\t,\v) & := \chi_{(0,\pi/2]}(\t) \bigg[\frac{\t\v}{(\t+\v)^2}\bigg]^{\kappa}
	(\pi-\v)^{\kappa} |\t-\v|^{2\s-1} \\
& \simeq \chi_{(0,\pi/2]}(\t) (\pi-\v)^{\kappa} \begin{cases}
	\t^{-\kappa+2\s-1} \v^{\kappa}, & \v \le \t/2 \\
	|\t-\v|^{2\s-1}, & \t/2 < \v < 2\t \\
	\t^{\kappa} \v^{-\kappa+2\s-1}, & \v \ge 2\t
	\end{cases}
\end{align*}
and the associated operator $\widetilde{\mathbb{T}}_6$. We will prove that $\widetilde{\mathbb{T}}_6$
has all the positive mapping properties claimed for $\mathbb{T}_6$ in the beginning of this proof.

To proceed, we consider $\v > 3\pi/4$ and $\v\le 3\pi/4$, and estimate $\widetilde{\mathbb{K}}_6(\t,\v)$
as follows:
\begin{align*}
\widetilde{\mathbb{K}}_6(\t,\v) & \lesssim \chi_{(3\pi/4,\pi)}(\v) (\pi-\v)^{\kappa} \begin{cases}
	\t^{-\kappa+2\s-1}, & \v \le \t/2 \\
	1, & \t/2 < \v < 2\t \\
	\t^{\kappa}, & \v \ge 2\t
	\end{cases} \\
& \quad + |\t-\v|^{2\s-1} + \chi_{\{\v<\t\}} \t^{-\kappa+2\s-1} \v^{\kappa} + \chi_{\{\v>\t\}}\t^{\kappa}
	\v^{-\kappa+2\s-1} \\
& \equiv \widetilde{\mathbb{K}}_{6,1}(\t,\v) + \widetilde{\mathbb{K}}_{6,2}(\t,\v)
	+ \widetilde{\mathbb{K}}_{6,3}(\t,\v) + \widetilde{\mathbb{K}}_{6,4}(\t,\v).	
\end{align*}
Let $\widetilde{\mathbb{T}}_{6,j}$, $j=1,\ldots,4$, be the corresponding integral operators.
We will analyze these operators separately. The mapping properties we shall verify will altogether
imply the desired conclusion about $\widetilde{\mathbb{T}}_6$. This implication will be perhaps best
seen by looking at the three cases coming from the comparison of $2\s-\kappa$ and $\kappa+1$, 
or equivalently $\s$ and $\kappa+1/2$, see Figures \ref{fig3} and \ref{fig4}.

Treatment of $\widetilde{\mathbb{T}}_{6,1}$ is straightforward. Indeed, we have
$$
\widetilde{\mathbb{K}}_{6,1}(\t,\v) \lesssim \chi_{(3\pi/4,\pi)}(\v)(\pi-\v)^{\kappa} \t^{\kappa}
	\lesssim \widetilde{\mathbb{K}}_{1}(\t,\v), 
$$
and consequently $\widetilde{\mathbb{T}}_{6,1}$ inherits the positive mapping properties of
$\widetilde{\mathbb{T}}_{1}$ verified above, see Figure \ref{fig3}. Furthermore, the kernel
$\widetilde{\mathbb{K}}_{6,2}(\t,\v)$ was already considered in the proof of Theorem \ref{thm:LpLqJac},
the analysis of $\mathcal{T}_6$ with $\alpha=\beta=-1/2$.
In particular, we know that $\widetilde{\mathbb{T}}_{6,2}$ possesses the positive
mapping properties shown for $\mathbb{T}_6$ in the proof of part (a), see Figure \ref{fig2}.

It remains to study $\widetilde{\mathbb{T}}_{6,3}$ and $\widetilde{\mathbb{T}}_{6,4}$.
Assuming that $p>1$ and $q<\infty$, we will verify the following items, which will complete proving
the positive results for $\mathbb{T}_6$.
\begin{itemize}
\item $\widetilde{\mathbb{T}}_{6,3}$ is of strong type $(p,q)$ if $\frac{1}p < 1+\kappa$ and
	$\frac{1}q \ge \frac{1}p -2\s$, and of restricted weak type $(\frac{1}{1+\kappa},q)$ for
	$\frac{1}q \ge 1+\kappa-2\s$.
\item $\widetilde{\mathbb{T}}_{6,4}$ is of strong type $(p,q)$ if $\frac{1}q > -\kappa$ and
	$\frac{1}q \ge \frac{1}p-2\s$, of weak type $(p,\frac{1}{-\kappa})$ for $\frac{1}p< 2\s-\kappa$,
	and of restricted weak type $(\frac{1}{2\s-\kappa},\frac{1}{-\kappa})$ in case $2\s-\kappa < 1$.
\end{itemize}

\noindent \textbf{Analysis of $\boldsymbol{\widetilde{\mathbb{T}}_{6,3}}$.}
We have
$$
\widetilde{\mathbb{T}}_{6,3}f(\t) = \t^{-\kappa+2\s-1}\int_0^{\t} \v^{\kappa} f(\v)\, d\v. 
$$
Assume that $\frac{1}p < 1+\kappa$, i.e.~$\kappa p'>-1$. By H\"older's inequality,
$$
\widetilde{\mathbb{T}}_{6,3}f(\t) \le \t^{-\kappa+2\s-1}\|f\|_p \bigg( \int_0^{\t} \v^{\kappa p'} \,d\v
	\bigg)^{1/p'} \lesssim \t^{2\s-1/p} \|f\|_p. 
$$
This estimate implies that $\widetilde{\mathbb{T}}_{6,3}$ is of strong type $(p,q)$ if
$(2\s-1/p)q>-1$, i.e.~$\frac{1}q > \frac{1}p -2\s$. Also, $\widetilde{\mathbb{T}}_{6,3}$ is
of weak type $(p,q)$ if $(2\s-1/p)q=-1$, i.e.~$\frac{1}q = \frac{1}p -2\s$. By interpolation,
$\widetilde{\mathbb{T}}_{6,3}$ is actually of strong type $(p,q)$ for $\frac{1}q = \frac{1}p-2\s$
(recall that we consider $q<\infty$).

It remains to check that $\widetilde{\mathbb{T}}_{6,3}$ is of restricted weak type $(\frac{1}{1+\kappa},q)$
for $\frac{1}q \ge 1+\kappa -2\s$. Let $E$ be a measurable subset of $(0,\pi)$. Clearly,
$$
\widetilde{\mathbb{T}}_{6,3}\chi_E(\t) \le \t^{-\kappa+2\s-1} \int_0^{|E|} \v^{\kappa}\, d\v
	\simeq \t^{-\kappa+2\s-1} |E|^{1+\kappa}.
$$
Consequently, 
$$
\widetilde{\mathbb{T}}_{6,3}\chi_E(\t) \lesssim \t^{-1/q} \|\chi_{E}\|_{1/(1+\kappa)}, 
$$
and now the conclusion easily follows.

\noindent \textbf{Analysis of $\boldsymbol{\widetilde{\mathbb{T}}_{6,4}}$.}
Let $p>1$ and $q<\infty$. Using H\"older's inequality, we get
$$
\widetilde{\mathbb{T}}_{6,4}f(\t) = \t^{\kappa} \int_{\t}^{\pi} \v^{-\kappa+2\s-1}f(\v)\, d\v
	\le \t^{\kappa}\|f\|_p \bigg( \int_{\t}^{\pi} \v^{(-\kappa+2\s-1)p'} \, d\v \bigg)^{1/p'}.
$$
To estimate the $L^{p'}$ norm here, we write
$$
\int_{\t}^{\pi} \v^{(-\kappa+2\s-1)p'}\, d\v \lesssim \begin{cases}
	1, & (-\kappa+2\s-1)p'>-1 \\
	\log\frac{\pi}{\t}, & (-\kappa+2\s-1)p'=-1 \\
	\t^{(-\kappa+2\s-1)p'+1}, & (-\kappa+2\s-1)p'<-1
	\end{cases}.
$$
Thus, for $\t \in (0,\pi)$,
$$
\widetilde{\mathbb{T}}_{6,4}f(\t) \lesssim \|f\|_p \begin{cases}
	\t^{\kappa}, & 2\s-\kappa > \frac{1}p \\
	\t^{\kappa}(\log\frac{\pi}{\t})^{1/p'}, & 2\s-\kappa = \frac{1}{p} \\
	\t^{2\s-1/p}, & 2\s-\kappa < \frac{1}p
	\end{cases}.
$$

Assume that $\frac{1}q > -\kappa$, i.e.~$\kappa q > -1$. If $\frac{1}p \le 2\s-\kappa$,
then $\widetilde{\mathbb{T}}_{6,4}$ is of strong type $(p,q)$. Indeed, in this case
$$
\widetilde{\mathbb{T}}_{6,4}f(\t) \lesssim \t^{\kappa} \Big(\log\frac{2\pi}{\t}\Big)^{1/p'}\|f\|_p,
$$
and the function $\t \mapsto \t^{\kappa} (\log\frac{2\pi}{\t})^{1/p'}$ is in $L^q$. Moreover,
$\widetilde{\mathbb{T}}_{6,4}$ is also of strong type $(p,q)$ if $\frac{1}p>2\s-\kappa$ and in addition
$\frac{1}q>\frac{1}p-2\s$, since then
$$
\widetilde{\mathbb{T}}_{6,4}f(\t) \lesssim \t^{2\s-1/p} \|f\|_p, 
$$
and the function $\t \mapsto \t^{2\s-1/p}$ belongs to $L^q$. Furthermore, for $\frac{1}p>2\s-\kappa$
and $\frac{1}q=\frac{1}p-2\s$, $\widetilde{\mathbb{T}}_{6,4}$ is of weak type $(p,q)$, in view of the bound
$$
\widetilde{\mathbb{T}}_{6,4}f(\t) \lesssim \t^{2\s-1/p}\|f\|_p = \t^{-1/q}\|f\|_p. 
$$
By interpolation, $\widetilde{\mathbb{T}}_{6,4}$ is in fact of strong type $(p,q)$ provided that
$\frac{1}p>2\s-\kappa$ and $\frac{1}q=\frac{1}p-2\s$.

Next, we verify the weak type $(p,\frac{1}{-\kappa})$ of $\widetilde{\mathbb{T}}_{6,4}$ for
$\frac{1}p< 2\s-\kappa$. This, however, is straightforward with the aid of the bound
$$
\widetilde{\mathbb{T}}_{6,4} f(\t) \lesssim \t^{\kappa} \|f\|_p. 
$$ 

Finally, $\widetilde{\mathbb{T}}_{6,4}$ is of restricted weak type $(\frac{1}{2\s-\kappa},\frac{1}{-\kappa})$
if $2\s-\kappa < 1$. Indeed, for measurable subsets $E$ of $(0,\pi)$,
$$
\widetilde{\mathbb{T}}_{6,4}\chi_E(\t) = \t^{\kappa}\int_{\t}^{\pi} \v^{-\kappa+2\s-1}\chi_{E}(\v)\,d\v
	\le \t^{\kappa} \int_0^{|E|} \v^{-\kappa+2\s-1}\, d\v \lesssim \t^{\kappa}\|\chi_E\|_{1/(2\s-\kappa)}
$$
uniformly in $\t\in (0,\pi)$, and the conclusion follows.

Now the desired positive results for $\mathbb{T}_6$ are justified. Passing to negative results,
we first ensure that $\mathbb{T}_6$ is not of restricted weak type $(p,q)$ if $\frac{1}q < \frac{1}p-2\s$.
Observe that
$$
\mathbb{K}_6(\t,\v) \simeq |\t-\v|^{2\s-1}, \qquad \t,\v \in (1/2,3/2).
$$
Thus we can invoke the arguments disproving the same mapping property for $\mathbb{T}_6$ in the
proof of part (a), see the analysis of $\mathcal{T}_6$, the case $\alpha=\beta=-1/2$ and hence $\delta=1/2$,
in the proof of Theorem \ref{thm:LpLqJac}.

Finally, we disprove the weak type $(\frac{1}{2\s-\kappa},\frac{1}{-\kappa})$ of $\mathbb{T}_6$ in case
$2\s-\kappa < 1+\kappa$. We may assume that $\alpha \le \beta$, so that $\kappa=\alpha+1/2$. Since
$$
\mathbb{K}_6(\t,\v) \ge \chi_{\{2>\v>2\t\}} \mathbb{K}_6(\t,\v) \gtrsim \chi_{\{2>\v>2\t\}}
	\t^{\kappa} \v^{-\kappa+2\s-1}, 
$$
we have
$$
\mathbb{T}_6 f(\t) \gtrsim \t^{\kappa} \int_{2\t}^2 \v^{-\kappa+2\s-1} f(\v)\, d\v, \qquad \t \in (0,1).
$$
Let $f(\v) = \chi_{(0,2)}(\v) \v^{\kappa-2\s}/\log\frac{\pi}{\v}$. Then $f\in L^{1/(2\s-\kappa)}$.
We will show that $\mathbb{T}_6 f \notin L^{-1/\kappa,\infty}$. Notice that
$$
\mathbb{T}_6 f(\t) \gtrsim \t^{\kappa} \int_{2\t}^{2} \frac{d\v}{\v\log\frac{\pi}{\v}} \equiv
	\t^{\kappa} h(\t), \qquad \t \in (0,1),
$$
where $h(\t)$ is a function increasing to $\infty$ as $\t$ decreases to $0$. Take $M>0$ arbitrarily large.
There exists $\eta>0$ such that $h(\t)>M$ for $\t \in (0,\eta)$. 
Consequently, for a fixed constant $c>0$ independent of $M$ and $\lambda > 0$,
$$
|\{\mathbb{T}_6 f > \lambda\}| \ge |\{c\t^{\kappa}h(\t)>\lambda\}| \ge
	|\{\t<\eta : c\t^{\kappa}M>\lambda\}| = \bigg|\bigg\{ \t < \eta : 
		\t < \Big(\frac{\lambda}{cM}\Big)^{1/\kappa} \bigg\}\bigg|.
$$
For $\lambda$ so large that $(\lambda/(cM))^{1/\kappa} < \eta$ is satisfied, we then get the lower bound
$$
|\{\mathbb{T}_6 f> \lambda\}| \ge \Big(\frac{cM}{\lambda}\Big)^{-1/\kappa}.
$$
Thus we see that
$$
\sup_{\lambda >0} \lambda |\{\mathbb{T}_6 f > \lambda\}|^{-\kappa} \ge cM
$$
and so the $L^{-1/\kappa,\infty}$ quasinorm of $\mathbb{T}_6 f$ cannot be finite.

The proof of part (b) in Theorem \ref{thm:LpLqJacL} and its sharpness is complete.
\end{proof}

\subsection{$\boldsymbol{L^p-L^q}$ estimates in the Fourier-Bessel settings}
We will give short proofs of Theorems \ref{thm:LpLqFB} and \ref{thm:LpLqFBL} by means of relating
the Fourier-Bessel potential kernels to the Jacobi potential kernels with suitably chosen parameters of type,
and then making use of the already proved results in the Jacobi settings.

\begin{proof}[{Proof of Theorem \ref{thm:LpLqFB}}]
Observe that, in view of Theorem \ref{thm:potkerFB} and Theorem \ref{thm:potker}, the potential kernel
in the natural measure Fourier-Bessel framework is controlled by the potential kernel in the Jacobi
trigonometric setting with parameters $\alpha=\nu$ and $\beta=-1/2$,
$$
\mathcal{K}_{\s}^{\nu}(x,y) \lesssim \mathcal{K}_{\s}^{\nu,-1/2}(\pi x,\pi y), \qquad x,y \in (0,1).
$$
Moreover, the corresponding measures are comparable,
$$
d\mu_{\nu}(x) \simeq d\mu_{\nu,-1/2}(\pi x), \qquad x \in (0,1).
$$
Thus we see that the Fourier-Bessel potential operator $\mathcal{I}_{\s}^{\nu}$ is controlled by the Jacobi
potential operator $\mathcal{I}_{\s}^{\nu,-1/2}$. Hence $\mathcal{I}_{\s}^{\nu}$ inherits all the
positive mapping properties of $\mathcal{I}_{\s}^{\nu,-1/2}$ stated in Theorem \ref{thm:LpLqJac}.
Consequently, the desired positive results for $\mathcal{I}_{\s}^{\nu}$ follow.

To verify the negative results, notice that
$$
\mathcal{K}_{\s}^{\nu}(x,y) \simeq \mathcal{K}_{\s}^{\nu,-1/2}(\pi x, \pi y), \qquad x,y \in (0,3/4).
$$
This comparability for $x$ and $y$ staying away from the right endpoint of $(0,1)$, together with 
the arguments given in the proof of Theorem \ref{thm:LpLqJac}, 
shows that $\mathcal{I}_{\s}^{\nu}$ has the same negative mapping
properties as those for $\mathcal{I}_{\s}^{\nu,-1/2}$ stated in Theorem \ref{thm:LpLqJac}.
The conclusion follows.
\end{proof}

\begin{proof}[{Proof of Theorem \ref{thm:LpLqFBL}}]
By \eqref{KlebFB}, Theorem \ref{thm:potkerFB}, Theorem \ref{thm:potker} and \eqref{Kleb}, we see that
the potential kernels in the Lebesgue measure Jacobi and the Lebesgue measure Fourier-Bessel settings are
comparable in the sense that
$$
\mathbb{K}_{\s}^{\nu}(x,y) \simeq \mathbb{K}_{\s}^{\nu,1/2}(\pi x, \pi y), \qquad x,y \in (0,1).
$$
Therefore the corresponding potential operators $\mathbb{I}_{\s}^{\nu}$ and $\mathbb{I}_{\s}^{\nu,1/2}$
possess exactly the same positive and negative $L^p-L^q$ mapping properties. Thus Theorem \ref{thm:LpLqFBL}
follows from Theorem \ref{thm:LpLqJacL} specified to $\alpha=\nu$ and $\beta = 1/2$.
\end{proof}

\section*{Appendix: summary of notation}

For reader's convenience, in Table 1 below we summarize the notation of various objects in
the contexts appearing in this paper, that is 
	\begin{itemize}
	 \item \textbf{Jacobi} \textbf{trig}onometric \textbf{pol}ynomial setting, 
	 \item \textbf{Jacobi} \textbf{trig}onometric \textbf{fun}ction setting,
	 \item \textbf{Jacobi} trigonometric function setting \textbf{scaled} to the interval $(0,1)$,
	 \item \textbf{nat}ural \textbf{meas}ure \textbf{F}ourier-\textbf{B}essel setting,
	 \item \textbf{Leb}esgue \textbf{meas}ure \textbf{F}ourier-\textbf{B}essel setting.
	\end{itemize}
{{
\begin{table*}[htbp]
\centering
\begin{tabular}{c|c|c|c|c|c|}
\cline{2-6}
 & Jacobi trig pol 
 & Jacobi trig fun
 & Jacobi scaled
 & FB nat meas
 & FB Leb meas \\ 
\cline{1-6} 
   \multicolumn{1}{|c|}{eigenfunctions} 
   & $\mathcal{P}_n^{\ab}$
   & $\phi_n^{\ab}$ 
   & $\widetilde{\phi}_n^{\ab}$
   & $\phi_n^{\nu}$
   & $\psi_n^{\nu}$ \\ 
\cline{1-6} 
   \multicolumn{1}{|c|}{reference measure} 
   & $d\mu_{\ab}$
   & $d\t$ 
   & $dx$ 
   & $d\mu_{\nu}$
   & $dx$ \\ 
\cline{1-6} 
   \multicolumn{1}{|c|}{`Laplacian'} 
   & $\mathcal{J}^{\ab}$
   & $\mathbb{J}^{\ab}$ 
   & $\widetilde{\mathbb{J}}^{\ab}$ 
   & $\mathcal{L}^{\nu}$
   & $\mathbb{L}^{\nu}$ \\ 
\cline{1-6} 
   \multicolumn{1}{|c|}{heat kernel} 
   & 
   & 
   & $\widetilde{\mathbb{G}}_t^{\ab}(x,y)$ 
   & 
   & $\mathbb{G}_t^{\nu}(x,y)$ \\ 
\cline{1-6} 
   \multicolumn{1}{|c|}{Poisson kernel} 
   & $\mathcal{H}_t^{\ab}(\t,\v)$
   & $\mathbb{H}_t^{\ab}(\t,\v)$ 
   & $\widetilde{\mathbb{H}}_t^{\ab}(x,y)$ 
   & $\mathcal{H}_t^{\nu}(x,y)$
   & $\mathbb{H}_t^{\nu}(x,y)$ \\ 
\cline{1-6} 
   \multicolumn{1}{|c|}{potential kernel} 
 & $\mathcal{K}_{\s}^{\ab}(\t,\v)$
   & $\mathbb{K}_{\s}^{\ab}(\t,\v)$ 
   & 
   & $\mathcal{K}_{\s}^{\nu}(x,y)$
   & $\mathbb{K}_{\s}^{\nu}(x,y)$ \\ 
\cline{1-6} 
   \multicolumn{1}{|c|}{potential operator} 
 & $\mathcal{I}_{\s}^{\ab}$
   & $\mathbb{I}_{\s}^{\ab}$ 
   & 
   & $\mathcal{I}_{\s}^{\nu}$
   & $\mathbb{I}_{\s}^{\nu}$ \\ 
\cline{1-6}
\noalign{\bigskip}
\end{tabular}
\caption{Summary of notation.}
\label{tab:notation}
\end{table*}
}}


\end{document}